\documentclass[12pt]{amsart}
\newcommand{\BA}{{\mathbb{A}}}

\newcommand{\BG}{{\mathbb{G}}}

\newcommand{\Id}{{\mathop{\operatorname{\rm Id}}}}

\usepackage{amssymb}
\usepackage[all]{xy}
\usepackage{mathrsfs}
\usepackage{enumerate}

\usepackage{amsmath}
\usepackage{amsxtra}
\usepackage{amscd}
\usepackage{amsthm}
\usepackage{amsfonts}
\usepackage{amssymb}
\usepackage{eucal}
\usepackage{graphicx}

\usepackage{bbm}

\DeclareMathOperator{\psId}{{Ps-Id}}

\DeclareMathOperator{\CT}{{CT}}
\DeclareMathOperator{\enh}{{enh}}
\DeclareMathOperator{\mes}{{mes}}

\newcommand{\secref}[1]{Sect.~\ref{#1}}
\newcommand{\lemref}[1]{Lemma~\ref{#1}}
\newcommand{\propref}[1]{Proposition~\ref{#1}}
\newcommand{\corref}[1]{Corollary~\ref{#1}}

\newcommand{\remref}[1]{Remark~\ref{#1}}

\newcommand{\CY}{{\mathcal{Y}}}

\newcommand{\bA}{{\mathbb A}}

\newcommand{\bC}{{\mathbb C}}
\newcommand{\bD}{{\mathbb D}}

\newcommand{\bF}{{\mathbb F}}
\newcommand{\bG}{{\mathbb G}}

\newcommand{\bQ}{{\mathbb Q}}
\newcommand{\bR}{{\mathbb R}}

\newcommand{\bX}{{\mathbb X}}

\newcommand{\bZ}{{\mathbb Z}}

\newcommand{\cA}{{\mathcal A}}
\newcommand{\cB}{{\mathcal B}}
\newcommand{\cC}{{\mathcal C}}
\newcommand{\cD}{{\mathcal D}}

\newcommand{\cF}{{\mathcal F}}

\newcommand{\cH}{{\mathcal H}}

\newcommand{\cL}{{\mathcal L}}
\newcommand{\cM}{{\mathcal M}}
\newcommand{\cN}{{\mathcal N}}

\newcommand{\cS}{{\mathcal S}}

\newcommand{\cY}{{\mathcal Y}}
\newcommand{\cZ}{{\mathcal Z}}

\newcommand{\sD}{{\mathscr D}}

\newcommand{\fC}{{\mathfrak C}}

\newcommand{\fL}{{\mathfrak L}}

\newcommand{\fR}{{\mathfrak R}}

\newcommand{\fl}{{\mathfrak l}}

\newcommand{\fs}{{\mathfrak s}}

\newcommand{\B}{{\cB}}

\newcommand{\Dmod}{\on{D-mod}}

\newcommand{\nc}{\newcommand}

\nc\wh{\widehat}

\nc\on{\operatorname}

\nc\Gr{\on{Gr}}

\nc\Fl{\on{Fl}}

\newtheorem{cor}[subsubsection]{Corollary}
\newtheorem{lem}[subsubsection]{Lemma}
\newtheorem{prop}[subsubsection]{Proposition}

\newtheorem{conj}[subsubsection]{Conjecture}
\newtheorem{thm}[subsubsection]{Theorem}

\theoremstyle{remark}
\newtheorem{rem}[subsubsection]{Remark}

\newcommand{\C}{{\cC}}
\newcommand{\Cp}{\cC_+}
\newcommand{\Cm}{\cC_-}

\DeclareMathOperator{\Bun}{{Bun}}

 \DeclareMathOperator{\Eis}{{Eis}}

\DeclareMathOperator{\red}{{red}} 
\DeclareMathOperator{\Sym}{{Sym}}

\DeclareMathOperator{\Cone}{{Cone}}

\newcommand{\limto}{{\displaystyle\lim_{\longrightarrow}}}
\newcommand{\rightlim}{\mathop{\limto}}

\newcommand{\leftlim}{\mathop{\displaystyle\lim_{\longleftarrow}}}
\newcommand{\limfromn}{\leftlim\limits_{\raise3pt\hbox{$n$}}}
\newcommand{\limton}{\rightlim\limits_{\raise3pt\hbox{$n$}}}

\newcommand{\epi}{\twoheadrightarrow}
\newcommand{\iso}{\buildrel{\sim}\over{\longrightarrow}}
\newcommand{\mono}{\hookrightarrow}

\DeclareMathOperator{\End}{{End}} \DeclareMathOperator{\Ext}{{Ext}}
\DeclareMathOperator{\Hom}{{Hom}} 
\DeclareMathOperator{\K}{\mathbf{K}} 

\DeclareMathOperator{\Pic}{{Pic}} 
\DeclareMathOperator{\Div}{{Div}}

\DeclareMathOperator{\Ker}{{Ker}} \DeclareMathOperator{\id}{{id}}
\DeclareMathOperator{\im}{{Im}} 
\DeclareMathOperator{\re}{{Re}}

\DeclareMathOperator{\Spec}{{Spec}}

\DeclareMathOperator{\Vect}{{Vect}}
\DeclareMathOperator{\diag}{{diag}}

\DeclareMathOperator{\Supp}{{Supp}}

\DeclareMathOperator{\I}{{I}}
\DeclareMathOperator{\Iw}{{Iw}}

\newcommand{\U}{\fL}
\newcommand{\W}{Q}

\theoremstyle{definition}
\newtheorem{ex}[subsubsection]{Example}
\newtheorem{defin}[subsubsection]{Definition}

\numberwithin{equation}{section}

\newcommand{\ql}{\overline{\bQ}_l}

\newcommand{\sfp}{{\mathsf{p}}}
\newcommand{\sfq}{{\mathsf{q}}}

\newcommand{\dK}{{\omega}}

\begin{document}

\title[A strange bilinear form on the space of automorphic forms]{On a strange invariant bilinear form on the space of automorphic forms}

\author{Vladimir Drinfeld}
\address{V.D.: Department of Mathematics,
University of Chicago, Chicago, IL 60637, USA}
\email{drinfeld@math.uchicago.edu}

\author{Jonathan Wang}
\address{J.W.: Department of Mathematics,
University of Chicago, Chicago, IL 60637, USA}
\email{jpwang@math.uchicago.edu}

\dedicatory{To Joseph Bernstein with deepest admiration}

%\date{}

\begin{abstract}
Let $F$ be a global field and $G:=SL(2)$. We study the bilinear form $\B$ on the space of $K$-finite smooth compactly supported functions on 
$G(\bA )/G(F)$ defined by 
$$\B (f_1,f_2):=\B_{naive}(f_1,f_2)-\langle M^{-1}\CT (f_1)\, ,\CT (f_2)\rangle,$$
where $\B_{naive}$ is the usual scalar product, $\CT$ is the constant term operator, and $M$ is the standard intertwiner. This form is natural from the viewpoint of the geometric Langlands program. To justify this claim, we provide a dictionary between the classical and `geometric' theory of automorphic forms. We also show that the form $\B$ is related to S.~Schieder's Picard-Lefschetz oscillators.
\end{abstract}



\maketitle

\section{Introduction}   \label{s:Intro}
\subsection{Some notation}
\subsubsection{}   \label{sss:some not}

Let $G$ denote the algebraic group $SL(2)$. Let $F$ be a global field (i.e., either a number field or a field finitely generated over $\bF_p$ of transcendence degree~1). Let $\bA$ denote the adele ring of $F$.

For any place $v$ of $F$, let $K_v$ denote the standard maximal compact subgroup of $G(F_v)$ (i.e., if $F_v=\bR$ then 
$K_v=SO(2)$, if $F_v=\bC$ then $K_v=SU(2)$, and if $F_v$ is non-Archimedean then $K_v=G(O_v)$, where $O_v\subset F_v$ is the ring of integers). Set $K:=\prod\limits_v K_v\,$; this is a maximal compact subgroup of $G(\bA )$.

\subsubsection{} \label{sss:E}
We fix a field $E$ of characteristic 0; if $F$ is a number field we assume that $E$ equals $\bR$ or $\bC$. Unless specified otherwise, all functions will take values in $E$.

\subsubsection{}   \label{sss:automorphic}
Let $\cA$ denote the space of $K$-finite $C^{\infty}$ functions
on $G(\bA )/G(F )$. (The letter $\cA$ stands for `automorphic'.) Let $\cA_c\subset\cA$ denote  the subspace of compactly supported functions.

\subsubsection{}
Fix a Haar measure on $G(\bA )$. If $f_1,f_2\in\cA$ and at least one of the functions $f_1,f_2$ is in $\cA_c\,$, we set
\begin{equation}  \label{e:Bnaive}
\B_{naive}(f_1,f_2)=\int\limits_{G(\bA)/G(F)}f_1(x)f_2(x)dx\, .
\end{equation}

\subsection{Subject of this article}
In this article we define and study an invariant\footnote{If $F$ is a function field then `invariant' just means invariance with respect to the action of $G(\bA)$. If $F$ is a number field then the notion of invariance is modified in the usual way, see formula \eqref{e:invarince of B} (a~modification is necessary because in this case $G(\bA)$ does not act on $\cA$).} symmetric bilinear form $\B$ on $\cA_c\,$, which is slightly different from $\B_{naive}\,$. (The definition of $\B$ will be given in Subsection~\ref{ss:def of B}). One has
\[
\B (f_1,f_2)=\B_{naive}(Lf_1,f_2),
\]
where $L:\cA_c\to\cA$ is a certain linear operator such that

(i) $Lf=f$ if $f$ is a cusp form;

(ii) the action of $L$ on an Eisenstein series has a nice description, see \propref{p:properties of A}(ii).

Let us note that $\B$ and $\B_{naive}$ slightly depend on the choice of a Haar measure on $G(\bA )$ but $L$ does not.

\begin{rem}
In this article we consider only $G=SL(2)$. However, we hope for a similar theory for any reductive $G$.

\subsection{Motivation}
Although the article is about automorphic forms in the most classical sense, the motivation comes from works \cite{DG2,G1}, which are devoted to the \emph{geometric} Langlands program. Let us explain more details.
\end{rem}

\subsubsection{A remarkable $l$-adic complex on $\Bun_G\times\Bun_G\,$}
Let $X$ be a geometrically connected smooth projective curve over a finite field $\bF_q\,$. Let $\Bun_G$ denote the stack of $G$-bundles on $X$. Let $\Delta:\Bun_G\to\Bun_G\times\Bun_G$ be the diagonal morphism. We have the $l$-adic complex $\Delta_*(\ql)$ on 
$\Bun_G\times\Bun_G\,$. 

Our interest in this complex is motivated by the fact that an analogous complex of  D-modules\footnote{More precisely, the complex $\Delta_!\omega_{\Bun_G}\,$, where $\Bun_G$ is the stack of $G$-bundles on a geometrically connected smooth projective curve over a field of characteristic 0. Analogy between $l$-adic sheaves and D-modules is discussed in Subsect.~\ref{ss:D&ell} of Appendix~\ref{s:miraculous}.} plays a crucial role in the theory of \emph{miraculous duality} on $\Bun_G\,$, which was developed in  \cite[Sect.~4.5]{DG2} and \cite{G1}. This theory tells us that the DG category of (complexes of) D-modules\footnote{Here we assume that the ground field has characteristic 0.} on $\Bun_G\,$ is equivalent to its Lurie dual (as predicted by the geometric Langlands philosophy), but the equivalence is defined in
a nontrivial way\footnote{The definition involves the complex $\Delta_!\omega_{\Bun_G}\,$, see Subsections~\ref{ss:psid}-\ref{ss:miraculous} of Appendix~\ref{s:miraculous}.}, and the fact the the functor in question is an equivalence is a highly nontrivial theorem \cite[Theorem~0.2.4]{G1}. More details on miraculous duality can be found in Subsection~\ref{ss:miraculous} of Appendix~\ref{s:miraculous}.

\subsubsection{The function $b$}   \label{sss:b}
Given $\cL_1,\cL_2\in\Bun_G (\bF_q)$, let $b(\cL_1,\cL_2)$ (resp. $b_{naive}(\cL_1,\cL_2)$) denote the trace of the geometric Frobenius acting on the stalk of the complex $\Delta_*(\ql)$ (resp. $\Delta_!(\ql)$) over the point $(\cL_1,\cL_2)\in(\Bun_G\times\Bun_G) (\overline{\bF}_q)$. It is clear that $b_{naive}(\cL_1,\cL_2)$ is just the number of isomorphisms between the $G$-bundles $\cL_1$ and $\cL_2$. 
Simon Schieder \cite[Prop.~8.1.5]{S} obtained the following  explicit formula for $b(\cL_1,\cL_2)$, in which the $SL(2)$-bundles $\cL_i$ are considered as rank 2 vector bundles:
\begin{equation}   \label{e:Schieder}
b(\cL_1,\cL_2)=b_{naive}(\cL_1,\cL_2)-\sum_f r(D_f),
\end{equation}
where $f$ runs through the set of  vector bundle morphisms $\cL_1\to\cL_2$ having rank 1 at the generic point of the curve, $D_f\subset X$ is the scheme of zeros of $f$, and for any finite subscheme $D\subset X$
\begin{equation} \label{e:r-Schieder}
r(D):=\prod_{x\in D_{\red}}(1-q_x).
\end{equation}

Here $q_x$ denotes the order of the residue field of $x$. 

\subsubsection{Relation between $\B$ and $b$}
Let $F$ be the field of rational functions on $X$. Then the quotient $K\backslash G(\bA )/G(F)$ identifies with 
$\Bun_G(\bF_q)$. So the functions $b$ and $b_{naive}$ from Subsect.~\ref{sss:b} can be considered as functions on $(G(\bA )/G(F))\times (G(\bA )/G(F))$. The following theorem is one of our main results. It will be proved in Subsect.~\ref{ss:proof-main}.

\begin{thm}   \label{t:our main}
As before, let $F$ be a function field.
Normalize the Haar measure on $G(\bA )$ so that $K$ has measure 1. Then for any $f_1,f_2\in\cA_c^K$ one has
\begin{equation}   \label{e:B}
\B (f_1,f_2)=\int\limits_{(G\times G)(\bA)/(G\times G)(F)}  b(x_1,x_2) f_1(x_1)f_2(x_2) dx_1 dx_2\,.
\end{equation}
\end{thm}

\begin{rem}   \label{r:Bnaive}
It is easy to see that in the situation of the theorem one has
\begin{equation}
\B_{naive}(f_1,f_2)=\int\limits_{(G\times G)(\bA)/(G\times G)(F)}  b_{naive}(x_1,x_2) f_1(x_1)f_2(x_2) dx_1 dx_2\,.
\end{equation}
\end{rem}

\begin{rem}   \label{r:hope}
We started working on this project by considering \eqref{e:B} as a temporary definition of $\B$. Thus $\B$ was defined only for function fields and only on the space of $K$-invariant functions from $\cA_c\,$, and the problem was to remove these two assumptions. The possibility of doing this was not clear \emph{a priori}, but formula~ \eqref{e:r-Schieder} gave some hope. Indeed, the key ingredient of this formula is the sequence $r_n=r_n(x)$ defined by
\[
r_0=1, \quad r_n=1-q_x \mbox{ if } n>0,
\]
and the good news is that 
\begin{equation}   \label{e:good news}
\sum\limits_{n=0}^\infty r_nq_x^{-ns}=\zeta_{F_x}(s)/\zeta_{F_x}(s-1),
\end{equation}
so the \emph{right-}hand side of \eqref{e:good news} makes sense \emph{even if $F_x$ is Archimedean.}
\end{rem}

\begin{rem}
The function $(\cL_1 ,\cL_2)\mapsto b(\cL_1 ,\cL_2)$ defined in Subsect.~\ref{sss:b} has the following property: for any closed point $x\in X$ one has 
\begin{equation}   \label{e:bHecke}
T_x^{(1)}(b)=T_x^{(2)}(b),
\end{equation}
where  $T_x^{(i)}$ denotes the Hecke operator with respect to $\cL_i\,$. This clearly follows from Theorem~\ref{t:our main} and $G(\bA )$-invariance of the form $\B$. On the other hand, it is not hard to deduce \eqref{e:bHecke} from the cohomological definition of $b$ given in Subsect.~\ref{sss:b}.
\end{rem}

\subsection{On the function spaces $\cA$, $\cA_c\,$,  and others}   \label{ss:apology}
The function spaces $\cA$ and $\cA_c\,$ introduced in Subsect.~\ref{sss:automorphic} are quite reasonable in the case that $F$ is a function field
(which is most important for us). If $F$ is a number field they are not really good (e.g., the space of cusp forms is not contained in $\cA_c$ in the number field case). The same is true for the function spaces $\C$, $\C_c\,$, and $\C_{\pm}\,$ introduced in \S\ref{ss:C} below.

The readers interested in the number field case are therefore invited to replace the spaces $\cA$, $\cA_c$, $\C$, $\C_c\,$, $\C_{\pm}$ by more appropriate ones.\footnote{E.g., it seems reasonable to replace $\cA_c$ by the space of all functions $f\in\cA$ such that $D(f)$ rapidly decreases (in the sense of \cite[Exercise 3.2.5]{Bu}) for every element $D$ of the universal enveloping algebra of the Lie $\bR$-algebra $\fs\fl(2,F\otimes\bR)$.} However, this is beyond the scope of this article.

\subsection{Structure of the article}
\subsubsection{A general remark}
In the main body of this article we work only with functions on $G(\bA )/G(F)$; sheaves appear only behind the scenes (e.g., as a source of the function $b$ from Subsect.~\ref{sss:b}).  But some strange definitions\footnote{E.g., the definitions of the form $\B$, the operator $\Eis'$, and the space $\cA_{ps-c}\,$.} from the main body of the article are motivated by works \cite{DG2,DG3, G1} on the \emph{geometric} Langlands program. 
This motivation is explained in Appendices~\ref{s:miraculous} and \ref{s: Eis_! and Eis'}.
 
\subsubsection{The main body of  the article}  \label{sss:main body}
In \secref{s:Eisenstein} we recall basic facts about the Eisenstein operator $\Eis$, the constant term operator $\CT$ and the `standard intertwiner' $M$ (which appears in the classical formula $\CT\circ\Eis=1+M$).  Let us note that \propref{p:invertibility} is possibly new and the `second Eisenstein operator' $\Eis':=\Eis\circ M^{-1}$ from Subsect.~\ref{ss:second Eis} is not quite standard (in standard expositions $\Eis'$ is hidden in the formulation of the functional equation for the Eisenstein series). Our decision to introduce $\Eis'$ as a separate object is motivated by Theorem~\ref{t: Eis_! and Eis'} from Appendix~\ref{s: Eis_! and Eis'}; this theorem establishes a relation between the operator $\Eis'$ and the functor $\Eis_!$ considered in works  \cite{DG3,G1} on the geometric Langlands program.

In \secref{s:def of B} we define and study the bilinear form $\B$ and the operator $L:\cA_c\to\cA\,$. The operator $M^{-1}$ plays a key role here. According to \propref{p:properties of A}, the  operator $L$ acts as identity on cusp forms; on the other hand, $L\circ\Eis=-\Eis'$. In Subsect.~\ref{ss:not positive} we show that  if $F$ is a function field the form $\B$ is not positive definite. (Most probably, this is so for number fields as well.)

 In \secref{s:ps-c} we introduce a subspace $\cA_{ps-c}\subset\cA$, where `ps' stands for `pseudo' and `c' stands for `compact support'. Roughly, 
 $\cA_{ps-c}$ consists of functions $f\in\cA$ such that the support of the constant term of $f$ is bounded in the `wrong' direction. 
 In the function field case we prove that the operator $L:\cA_c\to\cA$ induces an isomorphism $\cA_c\iso\cA_{ps-c}\;$, and we explicitly compute the inverse isomorphism and the bilinear form on $\cA_{ps-c}$ corresponding to $\B$. (Let us note that the result of this computation is used in Appendix~\ref{s:conjectural} as a heuristic tool.)
 
 In \secref{s:K-invariants} we describe the restriction of the intertwiner $M$ and its inverse to the subspace of $K$-invariants. The description is given in a format which is convenient for the proofs of Theorems~\ref{t:our main} and \ref{t: Eis_! and Eis'}. Let us note that the function 
 $\zeta_{F_x}(s)/\zeta_{F_x}(s-1)$ (which already appeared in \remref{r:hope}) plays a key role in  \secref{s:K-invariants} (see \propref{p:Mellin} and the proof of \propref{p:beta_v}).
 
In \secref{s:the form on K-invariants} we describe the restriction of the bilinear form $\cB$  to the subspace of $K$-invariants. 
In the function field case this description matches the r.h.s of formula~\eqref{e:B}, where $b$ is given by Schieder's formula \eqref{e:Schieder}; this gives a proof of Theorem~\ref{t:our main}. In the number field case we get a similar description in terms of Arakelov $G$-bundles, see 
Subsect.~\ref{sss:numb case}.

In \secref{s:invertibility} we prove the existence of $M^{-1}$ (this statement is used throughout the article).

\subsubsection{Appendices~\ref{s:miraculous}-\ref{s:conjectural}}
In Appendix~\ref{s:miraculous} we discuss a dictionary between the classical world of functions on $G(\bA )/G(F)$ and the non-classical world of D-modules on $\Bun_G$ considered in~\cite{DG2,DG3, G1} (or the parallel non-classical world of $l$-adic sheaves on $\Bun_G$). Then we use this 
dictionary and the results of \cite{DG2,G1} to motivate the definitions of the form $\B$ and the function $b$ from Subsection~\ref{sss:b}.

Let us recall an important difference between the world of functions and that of $l$-adic sheaves (or D-modules). For functions, there is only one type of pullback and one type of pushforward. For sheaves (or D-modules) one has \emph{four} functors (two pullbacks and two pushforwards). It is convenient  to group the four functors into two pairs: the pair of `right' functors\footnote{Each `right' functor is right adjoint to the corresponding `left' functor.} (i.e., $f^!$ and $f_*$) and that of `left' functors (i.e., $f^*$ and $f_!$). 

For us, the `right' functors are the main ones\footnote{Because for non-holonomic D-modules, the `left' functors are only partially defined.}, and in Subsect.~\ref{sss:Functions-sheaves} we redefine the functions-sheaves dictionary accordingly, so that the pullback and pushforward for functions correspond to the `right' functors $f^!$ and $f_*\,$. With this convention, the usual Eisenstein operator $\Eis$ corresponds to the `right' Eisenstein functor $\Eis_*$ (see Subsect.~\ref{sss:Eis*} for details). A more surprising part of the dictionary from Appendix~\ref{s:miraculous} is that  the `left' 
Eisenstein functor $\Eis_!$ (see Subsections~\ref{sss:Eis!}-\ref{sss:analog of Eis!} for details) is closely related to the `second Eisenstein operator' $\Eis'$ from Subsect.~\ref{ss:second Eis}.

The precise formulation of the above-mentioned relationship between $\Eis_!$ and $\Eis'$ is contained in Theorem~\ref{t: Eis_! and Eis'}.
Appendix~\ref{s: Eis_! and Eis'} is devoted to the proof of this theorem. Let us note that Appendix~\ref{s: Eis_! and Eis'} can be read independently of Appendix~\ref{s:miraculous}.

In Appendix~\ref{s:conjectural} we formulate a conjectural D-module analog of the elementary formula~\eqref{e:Ainverse}.

\subsection{A remark on Mellin transform}   \label{ss:avoiding Mellin}
The operator $\Eis$ is defined on a space of functions on $G(\bA )/T(F)N(\bA )$, where $N\subset G$ is a maximal unipotent subgroup and 
$T\simeq\bG_m$ is a maximal torus. We prefer not to decompose the space of such functions as a direct integral with respect to characters of
$T(\bA )/T(F )$. In other words, we avoid Mellin transform as much as possible. 

Here is one of the reasons for this. We prefer to do only those manipulations with functions that can be also done for $l$-adic sheaves and D-modules. On the other hand, in the setting of $l$-adic sheaves the Mellin transform on a torus  \cite{GL} or a similar functor on an abelian variety is not invertible.

\subsection{Relation with Bernstein's `second adjointness'} As already mentioned in Subsect.~\ref{sss:main body}, the inverse of the standard intertwiner $M$ plays a key role in this article. The operator $M$ has a local counterpart $M_v\,$, which is the Radon transform. The operator $M_v^{-1}$ is essentially the same as the `Bernstein map' introduced in \cite[Def.~5.3]{BK}; the precise meaning of the words `essentially the same' is explained in \cite[Theorem~7.6]{BK}. Let us also mention that the Bernstein map is studied in \cite{SV} (under the name of asymptotic map) in the more general context of spherical varieties.

In this paper (which is devoted to the case $G=SL(2)$) we do not use the machinery of \cite{BK,SV}. But probably this machinery will become necessary to treat an arbitrary reductive group $G$.

\subsection{Acknowledgements}
This article is strongly influenced by the ideas of J.~Bernstein and works by R.~Bezrukavnikov and D.~Gaitsgory (who are students of Bernstein) and S.~Schieder (a student of D.~Gaitsgory).

We thank S.~Schieder for informing us about his results. We thank J.~Arthur, J.~Bernstein, R.~Bezrukavnikov,  W.~Casselman, S.~Helgason, D.~Kazhdan, A.~Knapp, S.~Raskin, Y.~Sakellaridis, and N.~Wallach for valuable advice.

The research of V. D. was partially supported by NSF grants DMS-1001660 and DMS-1303100. The research of
J. W. was partially supported by the Department of Defense (DoD) through the NDSEG fellowship.

\section{Recollections on the Eisenstein and constant term operators}   \label{s:Eisenstein}
In this section we recall basic facts from the theory of Eisenstein series for $SL(2)$.
A detailed exposition can be found in \cite{Bu,Go, GS, Lan1, Lan2, MW} or  \cite[Sect.~16]{JL}).
Let us note that \propref{p:invertibility} is possibly new and the `second Eisenstein operator' from Subsect.~\ref{ss:second Eis} is not quite standard. 

\medskip

Recall that $G:=SL(2)$. Let $T\subset G$ denote the subgroup of diagonal matrices. 
Let $B\subset G$ be the subgroup of upper-triangular matrices and $N$ its unipotent radical.
Let $K\subset G(\bA )$ denote the maximal compact subgroup defined in Subsect.~\ref{sss:some not}.

We will identify $T$ with $\bG_m$ using the isomorphism
\begin{equation}   \label{e:coroot}
\BG_m\iso T, \quad t\mapsto \diag (t,t^{-1}):=\left(\begin{matrix} t&0\\0&t^{-1}\end{matrix}\right).
\end{equation}

\subsection{The degree map $\bA^{\times}/F^{\times}\to\bR\,$}  \label{ss:degree}
For $a\in\bA^{\times}/F^{\times}$ we set $\deg a:=-\log ||a||$, where the base of the logarithm is some fixed positive number greater than 1. If $F$ is a function field we always understand $\log$ as $\log_q\, $, where $q$ is the order of the field of constants of $F$.

\subsection{The degree map $K\backslash G(\bA )/T(F)N(\bA)\to\bR\,$}
\label{ss:adelic norm}
This is the unique map 
\begin{equation} \label{e:adelic norm}
\deg :K\backslash G(\bA )/T(F)N(\bA)\to\bR
\end{equation}
that takes $\diag  (a,a^{-1})$ to $\deg a$ for any $a\in\bA^{\times}$. The map \eqref{e:adelic norm} is continuous and proper.

\subsection{The spaces $\C$, $\C_c$, $\C_{\pm}\,$} \label{ss:C}
Let $\C$ denote the space of $K$-finite $C^{\infty}$ functions on $G(\bA )/T(F )N(\bA )$. Let $\C_c\subset\C$ stand for the subspace of compactly supported functions.

Given a real number $R$, let $\C_{\le R}\subset\C$ denote the set of all functions $f\in\C$ such that $f(x)\ne 0$ only if $\deg x\le R$ (here $\deg$ is understood in the sense of Subsection~\ref{ss:adelic norm}). Similarly, we have 
$\C_{\ge R}$, $\C_{> R}$, and so on.

Let $\Cm$ denote the union of the subspaces $\C_{\le R}$ for all $R$. Let $\Cp$ denote the union of the subspaces $\C_{\ge R}$ for all $R$.

Clearly $\Cm\cap\Cp=\C_c\,$, $\Cm+\Cp=\C\,$.

As already said in Subsect.~\ref{ss:apology}, in the number field case the spaces $\C$, $\C_c\,$, $\C_{\pm}$ are not quite reasonable. Nevertheless, we will work with them. 

\begin{ex}   \label{ex:functional}
Let $X$ be a geometrically connected smooth projective curve over a finite field $\bF_q\,$. Let $F$ be the field of rational functions on $X$. Let $O_{\BA}\subset\BA$ denote the subring of integral adeles. Since $G(\bA )=K\cdot B(\bA)$ the set $K\backslash G(\bA )/T(F)N(\bA)$ identifies with $T(O_{\bA})\backslash T(\bA)/T(F)$ and then (using the isomorphism \eqref{e:coroot}) with $O_{\bA}^{\times}\backslash \bA^{\times}/F^{\times}$, which is the same as the Picard group\footnote{Unless specified otherwise, in this article the symbol $\Pic$ denotes the Picard group (which is an \emph{abstract} group) rather than the Picard \emph{scheme}.} $\Pic X$. 
The space $\C^K$ identifies with the space of all functions on $\Pic X$. The space $\Cp^K$ (resp. $\Cm^K$) identifies with the space of functions $f$ on $\Pic X$ such that $f(\cM )=0$ if $\deg\cM\ll0$ (resp. if $\deg\cM\gg 0$).
\end{ex}

\subsection{Properties of the map $G(\bA )/B(F)\to G(\bA )/G(F )$}  
The following fact is well known and easy.

\begin{prop}   \label{p:deg+deg}
Suppose that $x,y\in G(\bA )/B(F)$ have the same image in $G(\bA )/G(F )$. If $x\ne y$ then $\deg x+\deg y\le 0$.
\end{prop}

The following well known fact is the main result of reduction theory. 

\begin{prop} \label{p:Minkowski}
There exists a number $R(F)$ with the following property: each point of $G(\bA )/G(F )$ has a pre-image in $G(\bA )/B(F)$ whose degree is  
$\,\ge -R(F)$.
\end{prop}

\subsection{The constant term operator}  \label{ss:CT}
Unless specified otherwise, we will always normalize the Haar measure on $N(\bA )$ so that $N(\bA )/N(F)$ has measure $1$. 

In Section~\ref{sss:automorphic} we defined the spaces $\cA$ and $\cA_c\subset\cA\,$.  
One has the \emph{constant term} operator $\CT:\cA\to\C$ defined by the formula
\begin{equation}
(\CT f)(x):=\int\limits_{N(\bA)/N(F)} f(xn)dn, \quad\quad f\in\cA\, , \, x\in G(\bA ).
\end{equation}
In other words, $\CT:\cA\to\C$ is the pull-push along the diagram
\begin{equation}  \label{e:1pull-push}
G(\bA )/G(F)\leftarrow G(\bA)/B(F)\to G(\bA)/T(F)N(\bA).
\end{equation}

It is well known that $\CT (\cA_c)\subset\Cm$ (this easily follows from \propref{p:deg+deg}).

\begin{ex}  \label{ex:CTfunctional}
Consider the situation of Example~\ref{ex:functional}. Then we saw that $\C^K$ identifies with the space of all functions on $\Pic X$. On the other hand, $\cA^K$ identifies with the space of all functions on $\Bun_G (\bF_q)$.
If $f$ is such a function and $\cM\in\Pic X$ then  $(\CT f)(\cM)$ is the average value of the pullback of $f$ to 
$\Ext (\cM^{-1},\cM)$ under the usual map $\Ext (\cM^{-1},\cM)\to\Bun_G$ (to an extension of $\cM^{-1}$ by $\cM$ one associates its central term). If $f$ has finite support then $(\CT f)(\cM)=0$ when $\deg\cM\gg 0$.
\end{ex}

\subsection{The Eisenstein operator}   \label{ss:Eis}
We define the \emph{Eisenstein operator}\footnote{The authors of \cite{MW} call it `pseudo-Eisenstein'.} 
$\Eis :\Cp\to\cA$ to be the pull-push along the diagram
\begin{equation}  \label{e:2pull-push}
G(\bA)/T(F)N(\bA)\leftarrow G(\bA)/B(F)\to G(\bA )/G(F)
\end{equation}
(to see that the pull-push makes sense, use \propref{p:deg+deg} combined with properness of the map 
$ G(\bA)/B(F)\to G(\bA)/T(F)N(\bA)$). Explicitly,
\begin{equation}
(\Eis \varphi)(x):=\sum\limits_{\gamma\in G(F)/B(F)} \varphi (x\gamma), \quad\quad  \varphi\in\Cp\, , \, x\in G(\bA ).
\end{equation}

It is easy to see that $\Eis (\C_c)\subset\cA_c\,$. 

\subsection{Duality between $\Eis$ and $\CT$}
Fix some Haar measure on $G(\bA )$. Combining it with the the Haar measure on $N(\bA )$ from Subsection~\ref{ss:CT}, we get an invariant measure on $G(\bA )/T(F)N(\bA )$ and therefore a pairing between $\Cm$ and $\Cp$ defined by
\begin{equation}  \label{e:pairing}
\langle \varphi_1\, ,\varphi_2\rangle:=\int\limits_{G(\bA )/T(F)N(\bA )} \varphi_1 (x)\varphi_2 (x) dx.
\end{equation}
We also have a similar pairing between $\C_c$ and $\C$.

On the other hand, we have the pairing between $\cA_c$ and $\cA$ denoted by $\B_{naive}$ and defined by \eqref{e:Bnaive}.

It is well known and easy to check that
\begin{equation}  \label{e:duality}
\langle\CT (f)\, ,\varphi\rangle=\B_{naive}(f, \Eis (\varphi ))
\end{equation}
if either $f\in\cA$ and $\varphi\in\C_c\,$, or  $f\in\cA_c$ and $\varphi\in\Cp\,$.

\subsection{The operator $M:\Cp\to\Cm\,$}  \label{ss:M}
Let $Y$ denote the space of pairs $(x_1,x_2)$, where $x_1,x_2\in G(\bA )/B(F)$ have equal image in $G(\bA )/G(F)$ and $x_1\ne x_2$. One has two projections $Y\to G(\bA )/B(F)$. Define $M:\Cp\to\Cm$ to be the pull-push along the diagram
\[
G(\bA )/T(F)N(\bA)\leftarrow G(\bA )/B(F)\leftarrow Y\to G(\bA )/B(F)\to G(\bA )/T(F)N(\bA).
\]
This makes sense by \propref{p:deg+deg}; moreover,  \propref{p:deg+deg} implies that for any number $R$ one has $M(\C_{\ge R})\subset \C_{\le -R}\, $.

The following explicit formula for $M:\Cp\to\Cm$ is well known:
\begin{equation}   \label{e:M}
(M\varphi )(x)=\int\limits_{N(\bA)}\varphi (xnw)dn, \quad\quad  \varphi\in\Cp\, ,\,  x\in G(\bA)/T(F)N(\bA ),
\end{equation}
where $w:=\left(\begin{matrix} 0&1\\-1&0\end{matrix}\right)\in SL(2)$.

Define an action of $\bA^{\times}/F^{\times}$ on $\C$ as follows: for $t\in\bA^{\times}/F^{\times}$ and $f\in\C$ we set
\begin{equation}   \label{e:t star f}
(t\star f)(x):=||t||^{-1}\cdot f(x\cdot \diag (t^{-1},t)), \quad\quad x\in G(\bA)/T(F)N(\bA ),
\end{equation}
where $\diag (t^{-1},t)$ is the diagonal matrix with entries $t^{-1},t$. Because of the $||t||^{-1}$ factor, this action preserves the scalar product \eqref{e:pairing}. One has
\begin{equation} \label{e:anticommutes}
M(t\star f)=t^{-1}\star Mf,  \quad\quad t\in \bA^{\times}/F^{\times}, f\in\C_+ \, .
\end{equation}

\subsection{The formula $\CT\circ\Eis=1+M$}   \label{ss:CTEis}
It is well known and easy to see that the composition  $\CT\circ\Eis :\Cp\to\C$ equals $1+M$, where $1$ denotes the identity embedding $\Cp\mono\C$ and $M$ is considered as an operator $\Cp\to\C$.

\subsection{The kernel of $\Eis :\C_c\to\cA_c\,$}  \label{ss;kernel}  

\begin{lem}   \label{l:KerEis}
Let $h\in\C_c\,$. Then $\Eis h=0$ if and only if $Mh=-h$.
\end{lem}

\begin{proof}
We have $\CT\circ\Eis=1+M$ (see Subsect.~\ref{ss:CTEis}). Therefore if $\Eis h=0$ then $(1+M)h= 0$.

To prove the converse, we can assume that $h$ takes values in $\bR$ (and even in $\bQ$ if $F$ is a function field). Then a positivity argument shows that to prove that 
$\Eis h=0$ it suffices to check that $\B_{naive} (\Eis h, \Eis h)=0$. But $$\B_{naive} (\Eis h, \Eis h)=\langle\CT\circ\Eis (h),h\rangle=
\langle(1+M)h,h\rangle=0$$ by formula~\eqref{e:duality}.
\end{proof}

\subsection{Invertibility of $M$}  \label{ss:invertibility}
\begin{prop}   \label{p:invertibility}
The operator $M:\Cp\to\Cm$ is invertible.
\end{prop}

We will prove this in Section~\ref{s:invertibility}. In fact, we will prove there a slightly stronger \propref{p:2invertibility}, which also says that
\begin{equation}   \label{e:continuity of inverse}
\mbox{if } \varphi\in\C_{\le -R} \mbox{ then } M^{-1}\varphi\in\C_{\ge R-a}\, , 
\end{equation}
where $a$ is a number depending only on the  non-Archimedean $K$-type of $\varphi$ (e.g., 
if $\varphi$ is $K_v$-invariant for each non-Archimidean place $v$ then one can take $a=0$).

Let us note that  a self-contained proof of invertibility of $M^K:\Cp^K\to\Cm^K$ is given in Section~\ref{s:K-invariants}.

\begin{rem}
Suppose that the field $E$ from Subsect.~\ref{sss:E} equals~$\bC$. Then $\Cp$ and $\Cm$ are LF-spaces (i.e., countable inductive limits of 
Fr\'echet spaces). The operator $M:\Cp\to\Cm$ is clearly continuous. By the open mapping theorem, this implies that $M^{-1}:\Cm\to\Cp$ is also continuous. On the other hand, continuity of $M^{-1}$ follows from 
\remref{r:continuity_automatic}.
\end{rem}

\subsection{The second Eisenstein operator} \label{ss:second Eis}
\subsubsection{Definition}
Define the \emph{second Eisenstein operator} 
$$\Eis':\Cm\to\cA$$ by
$\Eis':=\Eis\circ M^{-1}$. 
(A motivation will be given in Subsect.~\ref{sss:avatar} below.)

\begin{rem}   \label{r:CTEis'}
By Subsection~\ref{ss:CTEis}, the composition  $\CT\circ\Eis' :\Cm\to\C$ equals $1+M^{-1}$, where $1$ denotes the identity embedding $\Cm\mono\C$ and $M^{-1}$ is considered as an operator $\Cm\to\C$.
\end{rem}

\subsubsection{$\Eis'$ as an `avatar' of $\Eis$}    \label{sss:avatar}
The functional equation for Eisenstein series tells us that $\Eis'$ is an `avatar' of $\Eis$ in the sense of analytic continuation, just as the series $\sum\limits_{n\ge 0}z^n$ is an `avatar' of $\sum\limits_{n< 0}(-z^n)$.

To formulate a precise statement, let us assume that the field $E$ from Subsect.~\ref{sss:E} equals $\bC$.  Let $\varphi\in\C_c\,$.
For $t\in\bA^{\times}/F^{\times}$, define $h_t \, , h'_t\in\cA$ by
\[
h_t:=\Eis (t\star\varphi ) , \quad h'_t:=\Eis' (t\star\varphi ),
\]
where $t\star\varphi$ is defined by formula~\eqref{e:t star f}.
It is easy to check\footnote{To check the statement about $h'_t$, use \eqref{e:continuity of inverse}.} that for any fixed $g\in G(\bA )$ one has $h_t(g)=0$ if $||t||$ is small enough and $h'_t(g)=0$ if $||t||$ is big  enough. The theory of Eisenstein series tells us that $h_t$ and  $h'_t$ are related as follows: for any $g\in G(\bA )$ and any character $\chi :\bA^{\times}/F^{\times}\to\bC^{\times}$, the integral
\[
\int\limits_{t\in\bA^{\times}/F^{\times}} h_t(g)\chi (t) ||t||^{-s}dt
\]
absolutely converges if $\re s>1$, the integral
\[
\int\limits_{t\in\bA^{\times}/F^{\times}} h'_t(g)\chi (t) ||t||^{-s}dt
\]
absolutely converges if $\re s$  is sufficiently negative, and the functions of $s$ defined by these integrals extend to the same meromorphic function defined on the whole $\bC$.

\section{The bilinear form $\B$ and the operator $L$} \label{s:def of B}
\subsection{The form $\B$}   \label{ss:def of B}
Fix a Haar measure on $G(\bA )$. Then we have the form $\B_{naive}$ on $\cA_c$ and a pairing $\langle \, ,\rangle$ between $\Cp$ and $\Cm\,$, see formulas \eqref{e:Bnaive} and \eqref{e:pairing}. We also have a continuous linear operator $M^{-1}:\Cm\to\Cp\,$, see Section~\ref{ss:invertibility}.

\begin{defin}
For $f_1,f_2\in\cA_c$ set
\begin{equation}  \label{e:def of B}
\B (f_1,f_2):=\B_{naive}(f_1,f_2)-\langle M^{-1}\CT (f_1)\, ,\CT (f_2)\rangle \,.
\end{equation}
\end{defin}

The expression $\langle M^{-1}\CT (f_1)\, ,\CT (f_2)\rangle$ makes sense because $\CT (\cA_c)\subset\Cm\,$.

Note that $\B (f_1,f_2):=\B_{naive}(f_1,f_2)$ if $f_1$ or $f_2$ is cuspidal.

\subsection{The operator $L:\cA_c\to\cA\,$}  \label{ss:operator A}
One has the operators
\[
\cA_c\overset{\CT}\longrightarrow\Cm\overset{M^{-1}}\longrightarrow\Cp\overset{\Eis}\longrightarrow\cA\,.
\]
\begin{defin}
Define $L:\cA_c\to\cA$ by $L:=1-\Eis\circ M^{-1}\circ\CT$, where $1$ denotes the identity embedding $\cA_c\to\cA\,$.
\end{defin}

In other words, $L:=1-\Eis'\circ\CT$, where $\Eis'$ is the second Eisenstein operator defined in Subsection \ref{ss:second Eis}.

Note that unlike the form $\B$, the operator $L$ does not depend on the choice of a Haar measure on $G(\bA )$.

The relation between $\B$ and $L$ is as follows:
\begin{equation}   \label{e:B and A}
\B (f_1,f_2)=\B_{naive}(Lf_1,f_2)=\B_{naive}(f_1,Lf_2), \quad\quad f_1,f_2\in\cA_c\, .
\end{equation}
This is a consequence of formula \eqref{e:duality}.

Let $\cA_c^{cusp}$ denote the cuspidal part of $\cA_c\,$.

\begin{prop}    \label{p:properties of A}
(i)  If $f\in\cA_c^{cusp}$ then $Lf=f$.

(ii) For any $\varphi\in\C_c$ one has $L (\Eis\varphi)=-\Eis'\varphi$, where $\Eis'$ is the second Eisenstein operator defined in Subsection \ref{ss:second Eis}.
\end{prop}

\begin{proof}
(i) Cuspidality means that $\CT f=0$. In this case $Lf=f$ by the definition of~$L$.

(ii) Since $\CT\circ\Eis=1+M$ the composition $L\circ\Eis :\C_c\to\cA$ equals
$$\Eis-\Eis\circ M^{-1}\circ (1+M)=-\Eis\circ M^{-1}=-\Eis',$$
and we are done.
\end{proof}

\begin{rem}
If $F$ is a function field then $\cA_c=\cA_c^{cusp}\oplus\Eis(\C_c)$, so the operator $L:\cA_c\to\cA$ is uniquely characterized by properties (i)-(ii) from \propref{p:properties of A}. 
 \end{rem}
 
 \begin{rem}
If $F$ is a number field the previous remark does not apply  because the space $\cA_c^{cusp}$ is too small (possibly zero). As already said in Subsect.~\ref{ss:apology}, in the number field case it would be reasonable to replace $\cA_c$ by a bigger space and to extend the form $\B$ to the bigger space by continuity. But this is beyond the scope of this article. 
 \end{rem}
 
 Later we will show that the operator $L$ is injective and describe the inverse operator $\im L\to\cA_c$ (see \corref {c:A} and \propref{p:A}). In the case that $F$ is a function field we will also describe $\im L$ explicitly (see \corref {c:A} and Definition~\ref{d:Apsc}).

\subsection{Invariance of the form $\B$}
\subsubsection{The case that $F$ is a function field}
In this case the group $G(\bA )$ acts on $\cA$ and $\cA_c\,$. The operator $L$ clearly commutes with this action, so the form $\B$ is $G(\bA )$-invariant.

\subsubsection{General case}
 Now let $F$ be an arbitrary global field. Let $\cH_0$ denote the space of compactly supported distributions on $G(\bA)$ that are $K$-finite with respect to both left and right translations. Let $\cH$ denote the space of compactly supported distributions $\eta$ on $G(\bA )$ such that $\eta*\cH_0\subset\cH_0$ and $\cH_0*\eta\subset\cH_0\,$. Then $\cH$ is a unital associative algebra and $\cH_0$ is an ideal in $\cH\,$. The anti-automorphism of $G(\bA )$ defined by $g\mapsto g^{-1}$ induces an anti-automorphism of the algebra $\cH$, denoted by $\eta\mapsto\eta^{\star}$. It preserves $\cH_0\,$, and its square equals $\id_{\cH}\,$.
 
 The algebra $\cH$ acts on $\cA$ and $\cA_c$. The operator $L:\cA_c\to\cA$ commutes with the action of $\cH$ (because this is true for each of the operators $\Eis$, $M$, and $\CT$).
 
The form $\B_{naive}$ is invariant in the following sense:
 \[
 \B_{naive} (\eta*f_1,f_2)= \B_{naive} (f_1\, ,\eta^{\star}*f_2),\quad\quad \eta\in\cH\, , f_i\in\cA_c\, .
 \]
Since $L:\cA_c\to\cA$ commutes with the action of $\cH$ the form $\B$ has a similar invariance property:
 \begin{equation}     \label{e:invarince of B}
 \B (\eta*f_1,f_2)= \B (f_1\, ,\eta^{\star}*f_2),\quad\quad \eta\in\cH\, , f_i\in\cA_c\, .
 \end{equation}

\subsection{$\B$ is not positive definite}   \label{ss:not positive}
Suppose that the field $E$ from Subsect.~\ref{sss:E} equals~$\bR$. 
Then one can ask whether the form $\B$ is positive definite.  

If $F$ is a function field the situation is as follows. First of all,  the restriction of $\B$ to $\cA_c^{cusp}$ is positive definite (because it equals the restriction of $\B_{naive}\,$). On the other hand, the restriction of $\B$ to $\Ker (1+L)$ is \emph{negative} definite by formula~\eqref{e:B and A}. 
Let us prove that if $F$ is a function field then
\begin{equation}   \label{e:ker nonzero}
\Ker (\cA_c\overset{1+L}\longrightarrow\cA )\ne 0.
\end{equation}
(In fact, one can prove the following stronger\footnote{This is stronger than \eqref{e:ker nonzero} because the representation of $G(\bA )$ in 
$\cA_c$ is not admissible.} statement:  the representation of $G(\bA )$ in  $(1+L)(\cA_c)$  is admissible.)

By \propref{p:properties of A}(ii), $(1+L)\circ\Eis=\Eis-\Eis'$, so to prove \eqref{e:ker nonzero}, it suffices to show that
\begin{equation}   \label{e:not contained}
\Ker (\C_c\overset{\Eis-\Eis'}\longrightarrow\cA )\not\subset\Ker (\C_c\overset{\Eis}\longrightarrow\cA_c ).
\end{equation}
This follows from the next two lemmas.

\begin{lem}  \label{l:2KerEis}
If $F$ is a function field then $\Ker (\C_c\overset{\Eis}\longrightarrow\cA_c )\ne 0$.
\end{lem}

\begin{proof}
Given integers $a\le b$, let $G(\bA )_{[a,b ]}$ denote the set of all $x\in G(\bA )$ such that $a\le\deg x\le\ b$, where 
$\deg :K\backslash G(\bA )/T(F)N(\bA)\to\bZ$ is the map defined in Subsect.~\ref{ss:adelic norm}. For any integer $\cN\ge 0$, let $\C_{[-\cN,\cN]}\subset\C$ denote the set of all functions $f\in\C$ whose support is contained in $G(\bA )_{[-\cN ,\cN ]}/T(F)N(\bA )$. To prove the lemma, it suffices to show that
$\dim\Eis (\C_{[-\cN,\cN]}^K)<\dim\C_{[-\cN,\cN]}^K$ for $\cN$ big enough. It is easy to see that
\[
\dim\C_{[-\cN,\cN]}^K=|K\backslash G(\bA )_{[-\cN ,\cN ]}/T(F)N(\bA)|=(2\cN+1)\cdot |\Pic^0X|.
\]
So it suffices to prove that
\begin{equation}   \label{e:desired estimate}
\dim\Eis (\C_{[-\cN,\cN]}^K)\le \cN \cdot |\Pic^0X|+c
\end{equation}
for some $c$ independent of $\cN$. 

For any $f\in\C_{[-\cN,\cN]}$ the support of $\Eis (f)$ is contained in the image of $G(\bA )_{[-\cN ,\cN ]}/B(F)$ in $G(\bA )/G(F)$.
By Propositions~\ref{p:Minkowski}  and \ref{p:deg+deg},
\[
\im (G(\bA )_{[-\cN ,\cN ]}/B(F)\to G(\bA )/G(F))=\im (G(\bA )_{[-R,\cN ]}/B(F)\to G(\bA )/G(F)),
\]
where $R=R(F)$ is the number from Proposition~\ref{p:Minkowski}. So
\begin{equation}   \label{e:easy estimate}
\dim\Eis (\C_{[-\cN,\cN]}^K)\le |K\backslash G(\bA )_{[-R,\cN ]}/B(F)|. 
\end{equation}
On the other hand, it is easy to check and well known that if $g$ is the genus of $F$ then for any $\cN\ge g$ the map
\[
K\backslash G(\bA )_{[g,\cN ]}/B(F)\to K\backslash G(\bA )_{[g,\cN ]}/T(F)N(\bA )
\]
is bijective, so $|K\backslash G(\bA )_{[g,\cN ]}/B(F)|=(\cN-g+1)\cdot |\Pic^0X|$. Combining this with \eqref{e:easy estimate}, we get 
\eqref{e:desired estimate}.
 \end{proof}

\begin{lem}    \label{l:t star f}
Let $f\in\Ker (\C_c\overset{\Eis}\longrightarrow\cA_c )$ and $t\in\bA^{\times}/F^{\times}$. Define $t\star f\in\cA_c$ by formula \eqref{e:t star f}.
Then

(i) $t\star f\in\Ker (\C_c\overset{\Eis-\Eis'}\longrightarrow\cA )$;

(ii) if $\Eis (t\star f)=0$ and $\deg t\ne 0$ then $f=0$.
\end{lem}

\begin{proof}
By \lemref{l:KerEis}, $Mf=-f$. Recall that $M(t\star f)=t^{-1}\star Mf$. So
\begin{equation}   \label{e:Mtf}
M(t\star f)=-t^{-1}\star f.
\end{equation}

Let us prove statement (i). We have
\[
(\Eis -\Eis')(t\star f)=\Eis(t\star f)-\Eis\circ M^{-1} (t\star f)=\Eis(t\star f-t^{-1}\star M^{-1}f)=\Eis (t\star f+t^{-1}\star f).
\]
By \lemref{l:KerEis}, to show that $\Eis (t\star f+t^{-1}\star f)=0$,  it suffices to prove that $M(t\star f+t^{-1}\star f)=-(t\star f+t^{-1}\star f)$. This follows from formula~\eqref{e:Mtf} and a similar equality $M(t^{-1}\star f)=-t\star f$.

Let us prove statement (ii). By \lemref{l:KerEis}, if $\Eis (t\star f)=0$ then $M(t\star f)=-t\star f$. By \eqref{e:Mtf}, this means that 
$t\star f=t^{-1}\star f$. So the subset $\Supp f\subset G(\bA )/T(F)N(\bA )$ is stable under right multiplication by $\diag (t^2 ,t^{-2})$. But $\Supp f$ is compact. So if $\deg t\ne 0$ then $\Supp f=\emptyset$ and $f=0$.
\end{proof}

\section{The space $\cA_{ps-c}$  of `pseudo compactly supported' functions}   \label{s:ps-c}
In this section we define a subspace $\cA_{ps-c}\subset\cA$. In the case that $F$ is a function field we prove that $L$ induces an isomorphism $\cA_c\iso\cA_{ps-c}\,$; we also compute the inverse isomorphism, see formula~\eqref{e:Ainverse}. This formula is simpler than the formula for $L$ itself: it does not involve $M^{-1}$.

Using this isomorphism and the bilinear form $B$ on $\cA_c$ one gets a bilinear form on $\cA_{ps-c}$ in the function field case. \propref{p:form on ps-c} gives a simple explicit formula for the form on $\cA_{ps-c}\,$, which does not involve $M^{-1}$.

\subsection{The space $\cA_{\circ}$\,}   \label{ss:Acirc}
\begin{defin}  \label{d:Acirc}
$\cA_{\circ}$ is the space of all functions $f\in\cA$ such that $\CT (f)\in\Cm\,$.
\end{defin}

\begin{lem}   \label{l:Acirc}
(i) $\cA_{\circ}\supset\cA_c\,$.

(ii) If $F$ is a function field then $\cA_{\circ}=\cA_c\,$.
\end{lem}

\begin{proof}
We know that $\CT (\cA_c)\subset\Cm\,$. This is equivalent to (i).

It is well known that if $F$ is a function field then the kernel of $\CT:\cA\to\C$ (also known as the space of cusp forms) is contained in $\cA_c\,$. The usual proof of this statement (e.g., see \cite[Prop.~10.4]{JL}), in fact, proves the inclusion $\cA_{\circ}\subset\cA_c\,$. 
\end{proof}

\subsection{The space $\cA_{ps-c}\,$}  \label{ss:Aps-c}
Similarly to Definition~\ref{d:Acirc}, let us introduce the following one.
\begin{defin}  \label{d:Apsc}
$\cA_{ps-c}$ is the space of all functions $f\in\cA$ such that $\CT (f)\in\Cp\,$.
\end{defin}

Here `ps' stands for `pseudo'.

\subsection{The isomorphism $L:\cA_{\circ}\iso\cA_{ps-c}\,$} 
In Subsection~\ref{ss:Acirc} we defined a subspace $\cA_{\circ}\subset\cA$ containing $\cA_c\,$, which in the function field case is  equal to $\cA_c\,$. 
In Subsection~\ref{ss:operator A} we defined an operator $L:\cA_c\to\cA\,$ by the formula $$L:=1-\Eis\circ M^{-1}\circ\CT=1-\Eis'\circ \CT.$$ 
In fact, the same formula defines an operator
$\cA_{\circ}\to\cA$\,. By a slight abuse of notation, we will still denote it by $L$.

\begin{prop}  \label{p:A}
(i) For any $f\in\cA_{\circ}$ one has $Lf\in\cA_{ps-c}\,$.

(ii) The operator $L:\cA_{\circ}\to\cA_{ps-c}$ is invertible. For $g\in\cA_{ps-c}$ one has
\begin{equation}   \label{e:Ainverse}
L^{-1}g=g-\Eis\circ\CT (g).
\end{equation}
\end{prop}

Note that if $g\in\cA_{ps-c}$ then $\CT (g)\in\Cp$ by the definition of $\cA_{ps-c}\,$, so the expression $\Eis\circ\CT (g)$ makes sense.

\begin{proof}
(i) For any $f\in\cA_{\circ}$ one has the following identity in $\C\,$:
\begin{equation}    \label{e:1}
\CT (Lf)=\CT (f-\Eis'\circ \CT (f))=\CT (f)-(1+M^{-1})\circ\CT (f)=-M^{-1}\circ\CT (f)
\end{equation}
(the second equality follows from \remref{r:CTEis'}). So $\CT (Lf)\in\Cp\,$, which means that $Lf\in\cA_{ps-c}\,$.

(ii) For $g\in\cA_{ps-c}$ set $L'g:=g-\Eis\circ\CT (g)$. Then
\begin{equation}  \label{e:2}
\CT (L'g)=\CT (g)-(1+M)\circ\CT (g)=-M\circ\CT (g).
\end{equation}
In particular, $\CT (L'g)\in\Cm\,$, so $L'g\in\cA_{\circ}\,$. Thus $L'$ is an operator $\cA_{ps-c}\to\cA_{\circ}\,$.

Let us now check that $L'$ is inverse to $L:\cA_{\circ}\to\cA_{ps-c}\,$.

For any $g\in\cA_{ps-c}$ one has 
\[
LL'g=L'g-\Eis\circ M^{-1}\circ\CT (L'g),
\]
so formula \eqref{e:2} implies that $LL'g=g$.

For any $f\in\cA_{\circ}$ one has $L'Lf=Lf-\Eis\circ\CT (Lf)$, so formula \eqref{e:1} implies that $L'Lf=Lf+\Eis\circ M^{-1}\circ\CT (f)=f$.
\end{proof}

\begin{cor}    \label{c:A}
The map $L:\cA_c\to\cA$ is injective, and $L (\cA_c)\subset\cA_{ps-c}\,$. If $F$ is a function field then $L$ induces an isomorphism  
$\cA_c\iso\cA_{ps-c}\,$, whose inverse is given by formula~\eqref{e:Ainverse}.
\end{cor}

\begin{proof}
Use \propref{p:A} and \lemref{l:Acirc}(ii).
\end{proof}

\begin{prop}    \label{p:Eis'(C_c)}
The operator $\Eis':\Cm\to\cA$ maps $\C_c$ to $\cA_{ps-c}\,$.
\end{prop}

\begin{proof}
This immediately follows from Propositions~\ref{p:properties of A}(ii) and \ref{p:A}(i). 

On the other hand, here is a slightly more direct argument. By \remref{r:CTEis'},
\[
\CT (\Eis'(\C_c))=(1+M^{-1})(\C_c)\subset \C_c+\Cp=\Cp\, ,
\]
and the inclusion $\CT (\Eis'(\C_c))\subset\Cp$ means that $\Eis'(\C_c)\subset \cA_{ps-c}$ (by the definition of~$\cA_{ps-c}$).
\end{proof}

\subsection{The bilinear form on $\cA_{ps-c}\,$}
Now let us assume that $F$ is a \emph{function field.} Then $L$ induces an  isomorphism  
$\cA_c\iso\cA_{ps-c}\,$ (see \corref{c:A}). So one has the bilinear form on $\cA_{ps-c}$ defined by
\begin{equation}
\B _{ps}(g_1,g_2):=\B (L^{-1}g_1,L^{-1}g_2), \quad\quad g_i\in\cA_{ps-c}\,.
\end{equation}
We will write a simple formula for $\B _{ps}$ (see \propref{p:form on ps-c} below). It involves certain truncation operators.

\subsubsection{Truncation operators}   \label{sss:truncation}
Let $\cN\in\bR$.

Given a function $h\in\C$, define $h^{\le\cN}\in\C$ as follows: if the degree of $x\in G(\bA )/T(F)N(\bA )$ is $\le\cN$ then $h^{\le\cN}(x):=h(x)$, otherwise $h^{\le\cN}(x):=0$. Set $h^{>\cN}:=h-h^{\le\cN}$.

Given a function  $h\in\cA$, define $h^{\le\cN}\in\cA$ as follows: if all pre-images of $x$ in $G(\bA )/B(F)$ have degree $\le\cN$
then $h^{\le\cN}(x):=h(x)$, otherwise $h^{\le\cN}(x):=0$. Set $h^{>\cN}:=h-h^{\le\cN}$.

\begin{lem}     \label{l:truncation}
(i) Let $U\subset K$ be an open subgroup. Then there exists a number $\cN_0=\cN_0(U)\ge 0$ such that each element of 
$U\backslash G(\bA )/T(F)N(\bA )$ of degree\footnote{The map $\deg :U\backslash G(\bA )/T(F)N(\bA )\to\bZ$ is well-defined because $U\subset K$.} \,$>\!\cN_0$ has a single pre-image in $U\backslash G(\bA )/B(F)$.

(ii) Let $\cN_0$ be as in statement (i). Let $h\in\cA^U$ be such that  
$h=h^{>\cN_0}$. Then $\Eis (\CT(h)^{>\cN_0})=h$.
\end{lem}

\begin{proof}
Statement (i) is well known. In fact, if $U$ contains the principal congruence subgroup of $K$ of level $D$ then
one can take $\cN_0(U)=\max(0, g_F-1+\frac{1}{2}\cdot\deg D)$, where $g_F$ is the genus of $F$.

To prove (ii), consider the diagram
\begin{equation}   \label{e:correspondence}
(U\backslash G(\bA)/T(F)N(\bA))^{>\cN_0}\overset{p}\longleftarrow (U\backslash G(\bA)/B(F))^{>\cN_0}\overset{\pi}\longrightarrow U\backslash G(\bA )/G(F),
\end{equation}
where the superscript $>\cN_0$ means that we consider only elements of degree $>\cN_0$. In terms of diagram~\eqref{e:correspondence},
$\Eis (\CT(h)^{>\cN_0})=\pi_*p^*p_*\pi^*(h)$. Since $p$ is injective, $\pi_*p^*p_*\pi^*=\pi_*\pi^*$. Since $\cN_0\ge 0$ the map $\pi$ is injective by \propref{p:deg+deg}. Using this fact and the equality $h=h^{>\cN_0}$ we get $\pi_*\pi^*(h)=h$.
\end{proof}

\subsubsection{A formula for $\B_{ps}\,$}
As before, we assume that $F$ is a function field.

\begin{prop}   \label{p:form on ps-c}
Let $U\subset G(\bA )$ be an open subgroup. Then there exists a number $\cN_0(U)$ such that for any $g_1,g_2\in\cA_{ps-c}^U$ and any 
$\cN\ge\cN_0(U)$ one has
\begin{equation}  \label{e:form on ps-c}
\B_{ps}(g_1,g_2)=\B_{naive}(g_1^{\le\cN},g_2^{\le\cN})-\langle\CT(g_1)^{\le\cN},\CT(g_2)^{\le\cN}\rangle.
\end{equation}
Here $\langle\;,\;\rangle$ denotes the pairing \eqref{e:pairing}.
\end{prop}

\begin{rem}
The assumption $g_i\in\cA_{ps-c}$ means that $\CT (g_i)\in\Cp$. So $$\CT(g_1)^{\le\cN}\in \Cp\cap\Cm=\C_c\, .$$
Therefore the expression $\langle\CT(g_1)^{\le\cN},\CT(g_2)^{\le\cN}\rangle$ makes sense. Note that $g_i^{\le\cN}\in\cA_c$ by \propref{p:Minkowski}, so the expression $\B_{naive}(g_1^{\le\cN},g_2^{\le\cN})$ also makes sense.

\end{rem}

\begin{rem}
$\CT(g_i)^{\le\cN}$ is not the same as $\CT(g_i^{\le\cN})$ because $(\CT(g_i^{>\cN}))^{\le\cN}$ is usually nonzero.
\end{rem}

\begin{proof}
Without loss of generality, we can assume that $U\subset K$. Let $\cN_0(U)$ be as in \lemref{l:truncation}(i).
Let us show that \eqref{e:form on ps-c} holds for $\cN\ge\cN_0(U)$.

One has 
\[
\B_{ps}(g_1,g_2):=\B (L^{-1}g_1,L^{-1}g_2)=\B_{naive}(L^{-1}g_1,g_2). 
\]
By \eqref{e:Ainverse}, $L^{-1}g_1=g_1-\Eis\circ\CT (g_1)$. 
If $\cN\ge\cN_0(U)$ then $g_1^{>\cN}=\Eis (\CT (g_1)^{>\cN})$ by  \lemref{l:truncation}(ii). So
\[
g_1-\Eis\circ\CT (g_1)=g_1^{\le\cN}-\Eis (\CT (g_1)^{\le\cN}).
\]
Now using \eqref{e:duality}, we get
\begin{multline*}
\B_{ps}(g_1,g_2)=\B_{naive}(g_1^{\le\cN},g_2)-\langle\CT (g_1)^{\le\cN},\CT (g_2)\rangle\\
=\B_{naive}(g_1^{\le\cN},g_2^{\le\cN})-\langle\CT(g_1)^{\le\cN},\CT(g_2)^{\le\cN}\rangle,
\end{multline*}
and we are done. 
\end{proof}

\section{The action of $M$ and $M^{-1}$ on $K$-invariants}     \label{s:K-invariants}
Recall that $K$ denotes the standard maximal compact subgroup of $G(\bA )$. Let 
$$M^K:\Cp^K\to\Cm^K$$
denote the operator induced by $M:\Cp\to\Cm\,$. In this section we write explicit formulas for $M^K$ and $(M^K)^{-1}$ in a format which is convenient for the proofs of Theorems~\ref{t:our main} and \ref{t: Eis_! and Eis'}. 

First, we recall a well known description of $M^K$, see \lemref{l:M^K as conv}, formula~\eqref{e:alpha explicit}, and \lemref{l:alpha_v}. Then we deduce from it the description of $(M^K)^{-1}$ given by \corref{c:alpha invertible}, formula~\eqref{e:beta explicit}, and \propref{p:beta_v}; the key formulas are \eqref{e:nonArch beta}-\eqref{e:C beta}. In the case of function fields we slightly modify the description of $(M^K)^{-1}$ in Subsect.~\ref{ss: funct fields}.

\subsection{Some notation}   \label{ss:divisor notation}
Define a subset
$O_v\subset F_v$ and a subgroup 
$O_v^{\times}\subset F_v^{\times}$ by
\[
O_v:=\{x\in F_v: |x|\le 1\}, \quad O_v^\times :=\{x\in F_v^\times : |x|= 1\}.
\]
(Note that if $F_v$ is Archimedean then the subset $O_v\subset F_v$ is not a subring.)
Define a subset $O_\bA\subset\bA$ and a subgroup 
$O_\bA^{\times}\subset \bA^{\times}$ by
\[
O_\bA:=\prod_v O_v\, , \quad O_\bA^\times:=\prod_v O_v^\times\,.
\]

Sometimes we will use the notation $\Div (F):=\bA^{\times}/O_\bA^{\times}\,$. Elements of $\Div (F)$ will be called divisors (although in the number field case the precise name is \emph{Arakelov} divisor or \emph{replete} divisor). We have the closed submonoid $\Div_+ (F)\subset\Div (F)$ defined by
$$\Div_+ (F):=(\bA^\times\cap O_\bA ) /O_\bA^\times\, .$$
This is the submonoid of \emph{effective} divisors.

Similarly, for any place $v$ of $F$ we set 
$$\Div (F_v):=F_v^{\times}/O_v^{\times}, \quad \Div_+ (F_v):=(F_v^{\times}\cap O_v)/O_v^{\times}.$$

We will often use additive notation for divisors.

\subsection{A formula for $M^K$ in terms of convolution}
We identify $\C^K$ with the space 
$$C^{\infty}((T(\bA )\cap K)\backslash T(\bA )/T(F))=C^{\infty}(\bA^{\times}/(F^{\times}\cdot O_{\bA}^{\times}))=
C^{\infty}(\Div (F)/F^{\times}).$$
By definition, a function $\varphi\in C^{\infty}(\bA^{\times}/(F^{\times}\cdot O_{\bA}^{\times}))$ is in $\Cp^K$ 
if and only if 
$\varphi (x)=0$ for $\deg x\ll 0$ (i.e., for $||x||\gg 1$); similarly, $\varphi$ is  in $\Cm^K$ if and only if 
$\varphi (x)=0$ for $\deg x\gg 0$ (i.e., for $||x||\ll  1$).

In \lemref{l:M^K as conv} below we will write a formula for $M^K:\Cp^K\to\Cm^K$ in terms of convolution on the group 
$\bA^{\times}/O_{\bA}^{\times}\,$.

\subsubsection{A map $F_v\to F_v^\times/O_v^\times\,$}  \label{sss:kind-of-norm}
Let $v$ be a place of $F$. We have the Iwasawa decomposition $G(F_v)=K_v\cdot T(F_v)\cdot N(F_v)$.

Define a map $f_v:F_v\to F_v^\times/O_v^\times$ as follows: $f_v(x)$ is the class of any $a\in F_v^{\times}$ such that
\begin{equation}   \label{e: norm1}
\left(\begin{matrix} 1&x\\0&1\end{matrix}\right)\cdot
\left(\begin{matrix} 0&1\\-1&0\end{matrix}\right)
\in K_v\cdot\left(\begin{matrix} a^{-1}&0\\0&a\end{matrix}\right)
\cdot N(F_v).
\end{equation}
In other words, $f_v(x)$ describes the image in $K_v\backslash G(F_v)/N(F_v)$ of the matrix
\[
\left(\begin{matrix} 1&0\\-x&1\end{matrix}\right)=
\left(\begin{matrix} 0&-1\\1&0\end{matrix}\right)\cdot
\left(\begin{matrix} 1&x\\0&1\end{matrix}\right)\cdot
\left(\begin{matrix} 0&1\\-1&0\end{matrix}\right)
\]

\begin{lem}    \label{l:properties of f_v}
(i) $|f_v(x)|\le 1$ for all $x\in F_v\,$. 

(ii) The map $f_v:F_v\to F_v^\times/O_v^\times$  is proper.

(iii) Suppose that $F_v$ is non-Archimedean. Then $f_v (x)=1\Leftrightarrow x\in O_v\,$.
\end{lem}

\begin{proof}
Condition~\eqref{e: norm1} means that the norm\footnote{Here the meaning of the word `norm' depends on the type of the local field $F_v$ (e.g., if $F_v=\bR$ it means the Euclidean norm). On the other hand, one has the following uniform definition: the norm of a vector $(x_1,\ldots,x_n)\in F_v^n\setminus\{0\}$ is the $l^p$-norm of $(|x_1|,\ldots,|x_n|)\in\bR^n$, where $p=p(F_v):=[\bar F_v:F_v]$  and $|x_i|$ is the normalized absolute value.} of the vector $a\cdot (1,x)\in F_v^2$ equals 1. The lemma follows.
\end{proof}

\subsubsection{A map $\bA\to \bA^\times/O_\bA^\times\,$}   \label{sss:2kind-of-norm}
Let $f:\bA\to \bA^\times/O_\bA^\times=\Div (F)$ be the map induced by the maps $f_v:F_v\to F_v^\times/O_v^\times$ from Subsect.~\ref{sss:kind-of-norm}. In other words, for $x\in\bA$ one defines $f(x)$ to be the class of any $a\in\bA^{\times}$ such that
\[
\left(\begin{matrix} 1&x\\0&1\end{matrix}\right)\cdot
\left(\begin{matrix} 0&1\\-1&0\end{matrix}\right)
\in K\cdot\left(\begin{matrix} a^{-1}&0\\0&a\end{matrix}\right)
\cdot N(\bA ).
\]
The map $f:\bA\to \bA^\times/O_\bA^\times=\Div (F)$ is proper by \lemref{l:properties of f_v}(ii-iii).

\subsubsection{The measure $\alpha$}  \label{sss: measure on Div}
Let $f:\bA\to \Div (F)$ be as in Subsect.~\ref{sss:2kind-of-norm}. Define $\alpha$ to be the $f$-pushforward of the Haar measure on $\bA$ such that $\mes (\bA/F)=1$.
By \lemref{l:properties of f_v}(i), the measure $\alpha$ is supported on the submonoid 
$$\Div_+(F):=(\bA^\times\cap O_\bA  )/O_\bA^\times \, .$$
So we have an operator $\Cp^K\to\Cp^K$ defined by $\varphi\mapsto\alpha *\varphi$.

\begin{lem}   \label{l:M^K as conv}
For any $\varphi\in\Cp^K$ one has
\[
||a||^{-2}\cdot (M\varphi )(a^{-1})=(\alpha *\varphi )(a), \quad a\in\bA^{\times}/(F^{\times}\cdot O_{\bA}^{\times}).
\]
\end{lem}
The lemma follows straightforwardly from formula~\eqref{e:M}, which says that

\[
(M\varphi )(x)=\int\limits_{N(\bA)}\varphi (xnw)dn, \quad\quad  \varphi\in\Cp\, ,\,  x\in G(\bA)/T(F)N(\bA ),
\]
where $w:=\left(\begin{matrix} 0&1\\-1&0\end{matrix}\right)\in SL(2)$.

\subsection{The distribution $\alpha$ and its convolution inverse}   \label{ss:beta}
\subsubsection{An algebra of distributions}  \label{sss:A}
Let $A$ denote the space of distributions on $\Div (F)$ supported on $\Div_+ (F)$.
The map 
$$\Div_+(F)\times\Div_+(F)\to\Div_+(F)$$ 
induced by the group operation in $\Div (F)$ is proper, so $A$ is an algebra with respect to convolution. This algebra acts (by convolution) on $\Cp^K$.

\begin{rem}   \label{r:A for funct fields}
If $F$ is a function field then $A$ is just the completed semigroup algebra of the monoid $\Div_+(F)$. So in this case $A$ is a local ring; its maximal ideal consists of those distributions which are supported on $\Div_+(F)\setminus\{ 0\}$.
\end{rem}

\subsubsection{Invertibility statements}
The measure $\alpha$ from Subsect.~\ref{sss: measure on Div} is an element of the algebra $A$ from Subsect.~\ref{sss:A}.

\begin{prop}   \label{p:alpha invertible}
$\alpha$ is invertible in $A$.
\end{prop}

If $F$ is a function field the proposition immediately follows from \remref{r:A for funct fields} and the (obvious) fact that the support of $\alpha$ contains $0$.
In Subsections~\ref{sss:explicit alpha_v}-\ref{sss:beta} below we prove the proposition for any global field and give an explicit description of both 
$\alpha$ and its convolution inverse.

\begin{cor}   \label{c:alpha invertible}
The operator $M^K :\Cp^K\to\Cm^K$ is invertible. For any $u\in\Cm^K$ one has
\[
(M^K)^{-1} (u )=\beta *w\, ,
\]
where $\beta$ is the inverse of $\alpha$ in $A$ and $w\in\Cp^K$ is defined by
\[
w (x):=||x||^{-2}\cdot u (x^{-1}).
\]
\end{cor}

\subsubsection{The local measures $\alpha_v\,$}   \label{sss:local measures}
Let $v$ be a place of $F$. Equip  $F_v$ with the following Haar measure\footnote{This choice is dictated by the desire to have the simple formula~\eqref{e:zeta function}. Note that the same choice of the Haar measure on $\bC$ is made in \cite[2.5]{T1}, \cite[3.4.2]{De} and \cite[3.2.5]{T2}.}: if $F_v$ is non-Archimedean we require that 
$\mes (O_v)=1$; if $F_v=\bR$ we require that $\mes (\bR/\bZ)=1$; if $F_v\simeq\bC$ we require that $\mes (F_v/(\bZ+\bZ\cdot \sqrt{-1}))=2$.

Let $\alpha_v$ denote the pushforward of the above measure under the proper map 
$f_v:F_v\to F_v^\times/O_v^\times =\Div (F_v)$ from Subsect.~\ref{sss:kind-of-norm}. By \lemref{l:properties of f_v}(i), the measure $\alpha_v$ is supported on $\Div_+ (F_v)$. Then
\begin{equation}    \label{e:alpha explicit}
\mes (\bA/F)\cdot\alpha =\underset{v}\otimes\,\alpha_v\, ,
\end{equation}
where $\bA$ is equipped with the product of the Haar measures on the fields $F_v\,$.

\subsubsection{Explicit description of $\alpha_v\,$}    \label{sss:explicit alpha_v}
If $F_v$ is non-Archimedean we identify $\Div (F_v)$ with $\bZ$ using the standard valuation of $F_v^\times$. Then $\Div_+ (F_v)$ identifies with
$\bZ_{\ge 0}\,$.

If $F_v$ is Archimedean we identify $\Div (F_v)$ with $\bR^\times_{>0}$ using the normalized absolute value on $F_v$. Then $\Div_+ (F_v)$ identifies with the semi-open interval $(0,1]\subset\bR^\times_{>0}\,$. 

For any $t\in\bR$ and any $s$, let $(1-t^2)_+^s$ denote $(1-t^2)^s$ if $1-t^2>0$ and $0$ if $1-t^2\le 0$. The following lemma is straightforward.

\begin{lem}    \label{l:alpha_v}
(i) If $F_v$ is non-Archimedean then 
\[
\alpha_v=\delta_0+\sum_{n=1}^\infty  (q_v^n-q_v^{n-1})\cdot\delta_n\,
\]
where $\delta_n$ is the delta-measure at $n$ and $q_v$ is the order of the residue field.

(ii) If $F_v=\bR$ then $\alpha_v=2t^{-2}\cdot (1-t^2)_+^{-1/2}\cdot dt$, where $t$ is the coordinate on $\bR_{>0}\,$.

(iii) If $F_v\simeq\bC$ then $\alpha_v=2\pi t^{-2} \cdot H(1-t)dt$, where $H$ is the Heaviside step function, i.e., $H(x)=1$ if $x\ge 0$ and 
$H(x)=0$ if $x< 0$.
\end{lem}

\subsubsection{The Mellin transform of $\alpha_v\,$}
On $F_v^\times/O_v^\times$ one has the measure $\alpha_v$ and, for each $s\in\bC$, the function $x\mapsto |x|^s$, where $|x|$ denotes the normalized absolute value. Integrating this function against $\alpha_v$ one gets the \emph{Mellin transform} of $\alpha_v\,$.The following well known proposition immediately follows from \lemref{l:alpha_v} (in the case 
$F_v=\bR$ one uses the relation between the  $\rm B$-function and the $\Gamma$-function).

\begin{prop}   \label{p:Mellin}
Let $s\in\bC$, $\re s>1$. Then
\begin{equation}    \label{e:zeta function}
\int\limits_{x\in F_v^\times/O_v^\times} |x|^s\cdot\alpha_v= \zeta_{F_v}(s-1)/\zeta_{F_v}(s).
\end{equation}
Here $\zeta_{F_v}(s)$ is the local $\zeta$-function of $F_v\,$; in other words, 
$$\zeta_\bR (s):=\pi^{-s/2}\Gamma (s/2), \quad \zeta_\bC (s):=2(2\pi )^{-s}\Gamma (s),$$
and if $F_v$ is non-Archimedean then
$$\zeta_{F_v}(s):=(1-q_v^{-s})^{-1} .$$
\end{prop}

\subsubsection{The convolution inverse of $\alpha_v\,$}
\begin{prop}   \label{p:beta_v}
(i) There exists a (unique) distribution $\beta_v$ on $\Div (F_v)$ supported on $\Div_+ (F_v)$ such that $\alpha_v *\beta_v$ equals the $\delta$-measure at the unit of $\Div (F_v)$. 

(ii) The Mellin transform of $\beta_v$ equals  $\zeta_{F_v}(s)/\zeta_{F_v}(s-1)$, where $\zeta_{F_v}$ is defined in \propref{p:Mellin}.

(iii) If $F_v$ is non-Archimedean then 
\begin{equation}   \label{e:nonArch beta}
\beta_v=\delta_0+\sum_{n=1}^\infty  (1-q_v)\cdot\delta_n\,
\end{equation}
where $\delta_n$ is the delta-measure at $n$ and $q_v$ is the order of the residue field.

(iv) If $F_v=\bR$ then 
\begin{equation}   \label{e:R beta}
\beta_v=-\pi^{-1}\cdot (1-t^2)_+^{-3/2}\cdot t^{-1}dt\, , 
\end{equation}
where $t$ is the coordinate on $\bR_{>0}$ and $(1-t^2)_+^{-3/2}$ is regularized in the usual way\footnote{That is, one considers $(1-t^2)_+^s$ as a holomorphic function in the half-plane $\re s>-1$ with values in the space of generalized functions of $t$, then one extends this function meromorphically to all $s$, and finally, one sets $s=-3/2$. One can check that the scalar product of $(1-t^2)_+^{-3/2}\cdot dt$ with any smooth compactly supported function $h$ on $\bR_{>0}$ equals $\int_0^1( 1-t^2)^{-3/2}(h(t)-h(1))dt$.},
as explained in \cite[I.3.2]{Gelf}.

(v) If $F_v\simeq\bC$ then 
\begin{equation}  \label{e:C beta}
\beta_v=-(2\pi )^{-1}\delta'(t-1)\cdot t^{-1}dt\, .
\end{equation}
\end{prop}

\begin{proof} 
Define $\beta_v$ by one of the formulas~\eqref{e:nonArch beta}-\eqref{e:C beta}. It is easy to check that this $\beta_v$ has property (ii) (in the case
$F_v=\bR$ use the relation between the  $\rm B$-function and the $\Gamma$-function).

The Mellin transforms of both $\alpha_v$ and $\beta_v$ are defined if $\re s$ is big enough, and they are inverse to each other. So $\beta_v$ is inverse to $\alpha_v$ in the sense of convolution.
\end{proof} 

\subsubsection{Proof of \propref{p:alpha invertible}}   \label{sss:beta}
Let $\beta_v$ be as in \propref{p:beta_v}. Define a distribution $\beta$ on $\Div (F)$ by
\begin{equation}    \label{e:beta explicit}
\beta=\mes (\bA/F)\cdot (\underset{v}\otimes\,\beta_v)\, ,
\end{equation}
where $\bA$ is equipped with the product of the Haar measures on the fields $F_v\,$. By \eqref{e:alpha explicit}, $\beta$ is  inverse to $\alpha$ in the sense of convolution. \qed

\subsection{The operator $(M^K)^{-1}$ in the function field case}    \label{ss: funct fields}
In Subsect.~\ref{ss:beta} we gave a description of $(M^K)^{-1}$ for any global field $F$. Now we will make it slightly more explicit in the function field case.

Let $F$ be the field of rational functions on $X$, where $X$ is a geometrically connected smooth projective curve over $\bF_q$ of genus $g_X\,$. Then $\Div (F)=\Div (X)$. The set of closed points of $X$ will be denoted by $|X|$.

\subsubsection{The algebra $A$}   \label{sss:algebra A}
For $D\in\Div (X)$ let $\delta_D$ denote the corresponding $\delta$-measure on $\Div (X)$. In particular, we have $\delta_x$ for every $x\in |X|$.

By \remref{r:A for funct fields}, the algebra $A$ from Subsect.~\ref{sss:A} is the completed semigroup algebra of the monoid $\Div_+ (X)$ (or equivalently, the algebra of formal power series in $\delta_x\,$, $x\in |X|$).

\subsubsection{The element $\beta\in A$}
In \corref{c:alpha invertible} we defined an element $\beta\in A$. \propref{p:beta_v} and formula~\eqref{e:beta explicit} describe it explicitly.
Let us now reformulate this description slightly.

Note that if the Haar measure on $\bA$ is normalized by the condition $\mes (O_\bA )=1$ then $\mes (\bA/F )=q^{g_X-1}\,$. So from
\propref{p:beta_v}(ii) and formula~\eqref{e:beta explicit} one gets
\begin{equation}   \label{e:1beta}
\beta =q^{g_X-1}\cdot\U_0 * \U_1^{-1},
\end{equation}
where $\U_n\in A$ is defined by 
\begin{equation}  \label{e:Upsilon}
\U_n:=\prod_{x\in |X|}(1-q_x^n\cdot \delta_x)^{-1}=\sum_{D\in\Div_+(X)}q^{n\cdot\deg D}\cdot\delta_D\, .
\end{equation}
(The letter $\U$ is used here to remind of $L$-functions).

\subsubsection{The action of $A$ on $\Cp^K\,$}  
We identify  $\C^K$ with the space of functions on $\Pic X$, i.e., on the group of isomorphism classes of line bundles on $X$.
As already mentioned in Subsect.~\ref{sss:A}, the algebra $A$ acts on the subspace $\Cp^K\subset\C^K$ by convolution. The element $\delta_D\in A$ corresponding to 
$D\in\Div_+(X)$ acts on $\Cp^K$ as follows:
\begin{equation}   \label{e:action of delta}
(\delta_D*\varphi )(\cM )=\varphi (\cM (-D)), \quad \varphi\in\Cp^K\, ,\; \cM\in\Pic X\, .
\end{equation}

\subsubsection{Formulas for $(M^K)^{-1}$}  
Combining \corref{c:alpha invertible} and formula~\eqref{e:1beta}, we see that for any $u\in\Cm^K$ one has
\begin{equation}    \label{e:needed}
(M^K)^{-1} (u )=q^{g_X-1}\cdot\U_0 * \U_1^{-1} *w,
\end{equation}
where $w\in\Cp^K$ is defined by
\begin{equation}   \label{e:2needed}
w (\cM):=q^{2\cdot\deg\cM}\cdot u (\cM^{-1}).
\end{equation}

\section{The restriction of $\B$ to $\cA_c^K$}   \label{s:the form on K-invariants}
We keep the notation of Subsect.~\ref{ss:divisor notation}. 

\subsection{Haar measures on $G(\bA )$, $N(\bA )$, and $\Div (F)$}
\subsubsection{The Iwasawa map $G(\bA )/N(\bA)\to\Div (F)$} \label{sss:Iwasawa}
By this we mean the unique map 
\begin{equation}   \label{e:Iwasawa map}
\Iw :G(\bA )/N(\bA)\to\bA^\times/O_{\bA}^\times=\Div (F)
\end{equation}
such that $K\cdot\diag (a,a^{-1})$  goes to  $\bar a\in\bA^\times/O_{\bA}^\times$.

\subsubsection{The Haar measure $\nu_\gamma\,$}    \label{sss:nu}
If $\gamma$ is a Haar measure on $G(\bA )/N(\bA )$ then its direct image under the map \eqref{e:Iwasawa map} has the form $||x||^2\cdot dx$ for some Haar measure $dx$ on $\bA^\times/O_{\bA}^\times=\Div (F)$. This Haar measure on $\Div (F)$ will be denoted by $\nu_\gamma\,$.

\subsubsection{The Haar measure $\nu_{\mu/\mu'}\,$}   \label{sss:2nu}
Now let $\mu$ be a Haar measure on $G(\bA )$ and $\mu'$ a Haar measure on $N(\bA )=\bA$. They define a $G(\bA )$-invariant measure 
$\mu/\mu'$ on $G(\bA )/N(\bA )$ and therefore a Haar measure $\nu_{\mu/\mu'}$ on $\Div (F)$.

\begin{rem}    \label{r:comp}
Suppose that $F$ is a function field, $\mu$ is such that $\mes G(O_\bA)=1$, and $\mu'$ is such that $\mes O_\bA=1$. Then $\nu_{\mu/\mu'}\,$ is the standard Haar measure on the discrete group $\Div (F)$ (i.e., each element of $\Div (F)$ has measure 1).
\end{rem}

\subsection{The generalized function $r$ on $\Div (F)$}
Fix a Haar measure $\mu$ on $G(\bA )$. Then we have the bilinear form $\B$ on $\cA_c\,$. In \propref{p:non-naive} below we will write an explicit formula for its restriction to $\cA_c^K$. This description involves a certain generalized function $r$ on $\Div (F)$ depending on the choice of $\mu$.

Before reading the definition of $r$ given in Subsect.~\ref{sss:def r} below, the reader may prefer to have a look
at the very understandable statement of Subsect.~\ref{sss:our situation}.

\subsubsection{Definition of $r$}   \label{sss:def r}
Let $\mu'$ be the Haar measure on $\bA$ such that $\mes (\bA/F)=1$. By Subsect.~\ref{sss:2nu}, we have a Haar measure $\nu_{\mu/\mu'}$ on 
$\Div(F)$. Let $\beta$ be the distribution on $\Div (F)$ defined in \corref{c:alpha invertible}. Now define a generalized function $r$ on $\Div (F)$ by
\[
r:=\frac{\beta}{\nu_{\mu/\mu'}}\;.
\]

\subsubsection{Explicit description of $r$}      \label{sss:expl r}
 In Subsect.~\ref{sss:local measures} we fixed a Haar measure on each completion $F_v\,$. Their product defines a Haar measure 
 $\tilde\mu'$ on $\bA$. By \eqref{e:beta explicit}, we have
 \[
 \beta=\frac{\tilde\mu'}{\mu'}\cdot\underset{v}\otimes\beta_v\, ,
 \]
 where $\beta_v$ is the distribution on $\Div (F_v)$ defined by  \eqref{e:nonArch beta}-\eqref{e:C beta}. So one gets the following formula expressing $r$ as a tensor product of explicit local factors:
 \begin{equation}   \label{e:r-tilde}
 r=\frac{\tilde\beta}{\nu_{\mu/\tilde\mu'}}\; , \quad \tilde\beta:=\underset{v}\otimes\beta_v\;.
\end{equation}

\subsubsection{The function field case} \label{sss:our situation}
In this case $\Div (F)$ is discrete, so there is no difference between generalized functions on $\Div (F)$ and usual ones.
Suppose that the Haar measure $\mu$ on $G(\bA )$ is chosen so that $\mes K=1$. 
Then for any effective divisor $D\in \Div (F)$ (which is the same as a finite subscheme of $X$) one has
\begin{equation}   \label{e:r(D)-bis}
r(D)=\prod_{x\in D_{\red}}(1-q_x),
\end{equation}
and if $D$ is not effective then $r(D)=0$. This follows from formulas~\eqref{e:r-tilde} and \eqref{e:nonArch beta} combined with \remref{r:comp}.

\subsubsection{A remark on the number field case}
If $F$ is a number field then the Archimedean local factors $\beta_v$ from formula \eqref{e:r-tilde} are distributions which are \emph{not} locally integrable, see formulas ~\eqref{e:R beta}-\eqref{e:C beta}. So the generalized function $r$ is \emph{not} a usual one.

\subsection{The restriction of $\B$ to $\cA_c^K$}
Fix a Haar measure on $G(\bA )$. Then we have the bilinear forms $\B_{naive}$ and $\B$ on $\cA_c\,$ defined by \eqref{e:Bnaive} and 
\eqref{e:def of B}.  We also have the generalized function $r$ on $\Div (F)$ defined in Subsect.~\ref{sss:def r}.

\begin{prop}   \label{p:non-naive}
 For any $f_1,f_2\in\cA_c^K$ one has
 \begin{equation}   \label{e:non-naive}
\B_{naive}(f_1,f_2)-\B (f_1,f_2)=\int\limits_{(G\times G)(\bA)/H(F)}  r (\Iw (x_1)\cdot\Iw(x_2))f_1(x_1)f_2(x_2) dx_1 dx_2\,,
\end{equation}
where $\Iw :G(\bA )\to\bA^\times/O_{\bA}^\times=\Div (F)$ is the Iwasawa map (see Subsect.~\ref{sss:Iwasawa}) and $H\subset G\times G$ is the algebraic subgroup formed by
pairs $(b_1,b_2)\in B\times B$ such that $b_1b_2\in N$.
\end{prop}

\begin{rem}
Let us explain why the r.h.s. of \eqref{e:non-naive} makes sense. The generalized function
\begin{equation}   \label{e:before summation}
(x_1,x_2)\mapsto r (\Iw (x_1)\cdot\Iw(x_2)), \quad (x_1,x_2)\in (G\times G)(\bA)/H(F)
\end{equation}
is well-defined because the map $\Iw:K\backslash G(\bA )\to\Div (F)$ is a submersive map between $C^\infty$ manifolds (0-dimensional ones if $F$ is a function field). It remains to show that its support is proper over $(G\times G)(\bA)/(G\times G)(F)$. This follows from the inclusion 
$\Supp (r)\subset\Div_+(F)$, properness of the map $(G\times G)(\bA)/H(F)\to (\bA^\times /O_{\bA}^\times)\times\bR$ defined by 
$(x_1,x_2)\mapsto (\Iw (x_1)\cdot\Iw(x_2), \deg\Iw (x_1))$, and the fact that $\deg\Iw (x)$ is bounded above if the image of $x$ in $G(\bA )/G(F)$ belongs to a fixed compact.
\end{rem}

\begin{proof} 
Equip $\bA^\times$ with the Haar measure whose pushforward to $\bA^\times/O_{\bA}^\times$ equals the measure $\nu_{\mu/\mu'}$ from
Subsect.~\ref{sss:def r}.
 Then the r.h.s. of \eqref{e:non-naive} equals
\begin{equation}   \label{e:the-integral}
\int\limits_{(\bA^\times \times \bA^\times)/(F^\times )_{anti-diag}} r(xy)\cdot ||xy||^2\cdot\varphi_1(x)\varphi_2(y) dxdy, 
\end{equation}
where $\varphi_i:=\CT (f_i)\in\Cm^K$ and $(F^\times )_{anti-diag}:=\{ (t,t^{-1})\,|\,t\in F^\times\}$.

By \eqref{e:def of B}, the l.h.s. of \eqref{e:non-naive} equals
\[
\langle (M^K)^{-1}\varphi_1\, ,\varphi_2\rangle=\int\limits_{\bA^\times/F^\times} ((M^K)^{-1}\varphi_1)(y)\varphi_2(y)\cdot ||y||^2dy\, .
\]
By \corref{c:alpha invertible}, $(M^K)^{-1}\varphi_1=\beta*w$, where $w(x)=||x||^{-2}\cdot \varphi_1 (x^{-1})$; in other words,
\[
((M^K)^{-1}\varphi_1)(y)=\int\limits_{\bA^\times}r(xy)\cdot ||x||^2\cdot\varphi_1 (x)dx\, .
\]
So the l.h.s. of \eqref{e:non-naive} also equals \eqref{e:the-integral}, and we are done.
\end{proof}

\begin{cor}   \label{r:non-naive}
For any $f_1,f_2\in\cA_c^K$ one has
\begin{equation}   \label{e:2non-naive}
\B_{naive}(f_1,f_2)-\B (f_1,f_2)=\int\limits_{(G\times G)(\bA)/(G\times G)(F)} \cS (x_1,x_2)f_1(x_1)f_2(x_2) dx_1 dx_2\,.
\end{equation} 
where $\cS$ is the pushforward of the generalized function \eqref{e:before summation} under the natural map
$(G\times G)(\bA)/H(F)\to (G\times G)(\bA)/(G\times G)(F)$. \qed
\end{cor}

\subsection{Geometric interpretation}
The function $\cS$ from \corref{r:non-naive} is defined in terms of the diagram
\begin{equation}   \label{e:original diagram}
(G\times G)(\bA)/(G\times G)(F)\leftarrow (G\times G)(\bA)/H(F)\to\bA^\times/O_{\bA}^\times=\Div (F),
\end{equation}
in which the right arrow is the map $(x_1,x_2)\mapsto \Iw (x_1)\cdot\Iw(x_2)$. Let us give a geometric interpretation of this diagram.

\subsubsection{Matrices of rank $1$} Set $\bX:=(G\times G)/H$; this is an algebraic variety\footnote{The variety $\bX$ and its generalizations for arbitrary reductive groups (see \cite[Sect.~2.2]{BK}) play an important role in \cite{BK}.} equipped with an action of $G\times G$. Diagram~\eqref{e:original diagram} can be rewritten as
\begin{equation}     \label{e:modified diagram}
(G\times G)(\bA)/(G\times G)(F)\leftarrow (G\times G)(\bA)\underset{(G\times G)(F)}\times \bX (F)\to\Div (F)\, .
\end{equation}
We identify the $(G\times G)$-variety $\bX$ with the variety of $(2\times 2)$-matrices of rank $1$ via the map
\[ 
(g_1,g_2)\mapsto g_1\cdot \left(\begin{matrix} 0&1\\0&0\end{matrix}\right)\cdot g_2^{-1},\quad g_1,g_2\in G=SL(2).
\]

Let us describe the right arrow in \eqref{e:modified diagram}. For each place $v$, we have the `norm map'
\begin{equation} \label{e:norm map}
\nu_v:\bX (F_v)\to F_v^\times/O_v^\times=\Div (F_v);
\end{equation} 
namely, if $A$ is a  $(2\times 2)$-matrix over $F_v$ of rank $1$ then $\nu_v (A)$ is the class of any $a\in F_v^\times$ such that the operator $a^{-1}A$ has norm $1$.\footnote{In the Archimedean case one can use either the Hilbert-Schmidt norm or the operator norm (with respect to the Hilbert norm on $F_v^2$); on matrices of rank $1$ the two norms are the same.}  It is easy to check that the right arrow in \eqref{e:modified diagram} equals the composition of the action map $(G\times G)(\bA)\underset{(G\times G)(F)}\times \bX (F)\to\bX (\bA )$ and the map $\nu :\bX (\bA )\to\Div (F)$ obtained from the maps \eqref{e:norm map}.

\subsubsection{Function field case} \label{sss:funct case}
Let $F$ be a function field, and let $X$ be the corresponding connected smooth projective curve over a finite field. 
The function $\cS$  from \corref{r:non-naive} is $K$-invariant, so one can consider it as a function on the set of isomorphism classes of pairs 
$(\cL_1,\cL_2)$, where $\cL_1$ and $\cL_2$ are rank 2 vector bundles on $X$ with trivialized determinants.

\begin{prop}  \label{p:S-r}
In this situation
\begin{equation}    \label{e:S-r}
\cS (\cL_1,\cL_2)=\sum_f r(D_f),
\end{equation}
where $f$ runs through the set of  morphisms $\cL_1\to\cL_2$ of  generic rank 1 and $D_f$ is the divisor of zeros of $f$.
\end{prop}

\begin{proof}
By definition, $\cS$ is obtained from $r$ using pull-push along diagram~\eqref{e:modified diagram}. If $y$ is a point of  
$(G\times G)(\bA)/(G\times G)(F)$ corresponding to $(\cL_1 ,\cL_2)$ then the pre-image of $y$ in 
$(G\times G)(\bA)\underset{(G\times G)(F)}\times \bX (F)$ identifies with the set of rational sections of the bundle $\bX_{\cL_1,\cL_2}\to X$ associated to the $(G\times G)$-torsor $(\cL_1,\cL_2)$ and the $(G\times G)$-variety $\bX$. If one thinks of $\cL_1$, $\cL_2$ as rank 2 vector bundles with trivialized determinants then a rational section of $\bX_{\cL_1,\cL_2}$ is the same as a rational morphism $f:\cL_1\dashrightarrow \cL_2$ of (generic) rank 1. The right arrow in diagram~\eqref{e:modified diagram} associates to such $f$ the divisor $D_f$ whose multiplicity at $x\in X$ is the order of zero of $f$ at $x$ (as usual, the order of zero is negative if $f$ has a pole at $x$). Finally, since $\Supp r\subset\Div_+(F)$ only true morphisms $f:\cL_1\to\cL_2$ contribute to $\cS$.
\end{proof}

\subsubsection{Number field case} \label{sss:numb case}
Let $F$ be a number field and $O_F$ its ring of integers. Then formula \eqref{e:S-r} remains valid after the following modifications. First, $\cL_i$ is now an \emph{Arakelov} $SL(2)$-bundle, i.e., a rank 2 vector bundle on $\Spec O_F$ with trivialized determinant and with a Euclidean/Hermitian metric on $\cL_i\otimes_{O_F}F_v$ for each Archimedean place $v$ (these metrics should be compatible with the trivialization of the determinant). Second, $D_f$ is now an \emph{Arakelov} divisor whose non-Archimedean part is the scheme of zeros of $f$ (which is an effective divisor on $\Spec O_F$), and whose
Archimedean part is the collection of all Archimedean norms of~$f:\cL_1\to\cL_2\,$. Finally, both sides of \eqref{e:S-r} are now \emph{generalized} functions on the $C^\infty$-stack\footnote{For the notion of $C^\infty$-stack see \cite{BX}.} of Arakelov $(G\times G)$-bundles.

\subsection{Proof of Theorem~\ref{t:our main}}  \label{ss:proof-main}
Let $F$ be a function field, and let $X$ be the corresponding connected smooth projective curve over a finite field. 
In Theorem~\ref{t:our main} the Haar measure on $G(\bA )$ is normalized by the condition $\mes K=1$, so the function $r$ on $\Div (X)$ is given by formula~\eqref{e:r(D)-bis}. Comparing formula~\eqref{e:S-r} with Schieder's formulas~\eqref{e:Schieder}-\eqref{e:r-Schieder}, we see that 
$\cS=b_{naive}-b$, where $b$ and $b_{naive}$ are as in Subsect.~\ref{sss:b}. 
By \remref{r:Bnaive}, this is equivalent to Theorem~\ref{t:our main}.  \qed

\section{Invertibility of the operator $M:\Cp\to\Cm$}  \label{s:invertibility}
The goal of this section is to prove \propref{p:2invertibility}, which says that the operator $M:\Cp\to\Cm$ is invertible (and a bit more). We deduce it from the corresponding local statement (Proposition~\ref{p:local invertibility}). To do this, we use a slightly nonstandard variant of infinite tensor products (see Subsect.~\ref{ss:infinite tensor}); it is here that we need the algebras introduced in Subsections~\ref{ss:global ring R} and \ref{sss:local ring R}. Philosophically, introducing the algebra $\fR^\mu$ from Subsections~\ref{ss:global ring R} and its actions is a substitute for Mellin transform with respect to $T(\bA )/T(F )$, which we tried to avoid (see Subsect.~\ref{ss:avoiding Mellin}).

\medskip

We will assume that the field $E$ from Subsection~\ref{sss:E} equals~$\bC$ (this  assumption is harmless).

\subsection{Notation and conventions}
\subsubsection{}   \label{sss:the number a}
We fix some real number $A>1$. For $x\in\BA^\times$ we set $\deg x:=-\log_A ||x||$. Define $\deg :K\backslash G(\bA )/T(F)N(\bA)\to\bR$ just as in Subsect.~\ref{ss:adelic norm}.

\subsubsection{}
Recall that $\C_+$ denotes the union of the spaces $\C_{\ge Q}$, where $\C_{\ge Q}$ is the space of $K$-finite $C^\infty$ functions 
$f:G(\bA )/T(F)N(\bA)\to\bC$ such that
\begin{equation}   \label{e:Supp}
\Supp (f)\subset \{ x\in G(\bA )/T(F)N(\bA)\,|\,\deg x\ge Q\}.
\end{equation}
Similarly, we have the spaces $\C_{\le Q}$ and their union $\C_-\,$.

Let $\fC_{\ge Q}$ denote the space of $K$-finite \emph{generalized} functions\footnote{If one fixes a $G(\bA)$-invariant measure on 
$G(\bA)/T(F)N(\bA)$ then a generalized function is the same as a distribution.} $f$ on $G(\bA )/T(F)N(\bA)$ satisfying \eqref{e:Supp}. Clearly 
$\C_{\ge Q}\subset\fC_{\ge Q}\,$, and if $F$ is a function field then $\C_{\ge Q}=\fC_{\ge Q}\,$. Quite similarly, define $\fC_{\le Q}$ and $\fC_\pm\,$.

The operator $M:\C_+\to\C_-$ defined in Subsection~\ref{ss:M} naturally extends to an operator $\fC_+\to\fC_-\,$, which will still be denoted by 
$M$.

\subsubsection{}  \label{sss:change}
We will slightly \emph{change} our conventions regarding the action of $\bA^\times/F^\times$. Namely, the action of  
$\bA^\times/F^\times$ on $\C_+\, , \fC_+$ will still be defined by formula \eqref{e:t star f}, but the action of $\bA^\times/F^\times$ on $\C_-\, , \fC_-$ will be the opposite one.  Then the operator $M$ is  $(\bA^\times/F^\times)$-equivariant.

\subsubsection{}
Just as in Subsection~\ref{ss:divisor notation}, we use the notation $O_v^\times :=\{x\in F_v^\times : |x|= 1\}$ (even if $v$ is Archimedean)
 and the notation $O_\bA^\times:=\prod\limits_v O_v^\times$ (even if $F$ is a number field).

It is convenient to fix a character $\mu :O_\bA^\times/(O_\bA^\times\cap F^\times )\to\bC^\times$ once and for all. Let $\C_+^\mu\subset\C_+$ denote the $\mu$-eigenspace for the $O_\bA^\times$-action. Similarly, we have $\C_-^\mu$ and $\fC_\pm^\mu\,$.

\subsection{The algebra $\fR^\mu$}  \label{ss:global ring R}
Let $\fR_{\ge Q}$ denote the space of distributions $\eta$ on $\bA^\times/F^\times$ such that $\deg x\ge Q$ for all $x\in\Supp\eta$.
Let $\fR$ denote the union of $\fR_{\ge Q}$ for all $Q\in\bR$. Then $\fR$ is a filtered algebra with respect to convolution.

Let $e^\mu\in\fR$ be the product of $\mu^{-1}$ and the normalized Haar measure on the compact subgroup 
$O_\bA^\times/(O_\bA^\times\cap F^\times )\subset\bA^\times/F^\times$. Clearly $e^\mu$ is an idempotent.

Set $\fR^\mu:=e^\mu\cdot\fR$; this is a unital algebra. 
Let $R^\mu\subset\fR^\mu$ denote the ideal formed by smooth measures. If $F$ is a function field then $R^\mu =\fR^\mu$. 
One has
\[
 \fR^\mu_{\ge Q}\cdot\fC^\mu_{\ge Q'}\subset\fC^\mu_{\ge Q'+Q}\, , \quad \fR^\mu_{\ge Q}\cdot\fC^\mu_{\le Q'}\subset\fC^\mu_{\le Q'-Q}\, . 
\]

\subsection{Structure of the $\fR^\mu$-modules $\fC_\pm^\mu$ and $\C_\pm^\mu\,$.}
The action of $\bA^\times/F^\times$ induces an action of $\fR^\mu$ on $\fC_\pm^\mu$ and 
$\C_\pm^\mu\,$. Let us describe the structure of $\fC_\pm^\mu$ and $\C_\pm^\mu$ as $\fR^\mu$-modules equipped with $K$-action.

First, we have
\[
\fC_\pm^\mu=\bigoplus\limits_\rho \fC_\pm^{\mu,\rho}, \quad \C_\pm^\mu=\bigoplus\limits_\rho \C_\pm^{\mu,\rho},
\]
where $\rho$ runs through the set of isomorphism classes of irreducible $K$-modules, $\fC_\pm^{\mu,\rho}:=\rho\otimes\Hom_K (\rho ,\fC_\pm)$, 
$\C_\pm^{\mu,\rho}:=\rho\otimes\Hom_K (\rho ,\C_\pm)$. 

Now let $V_\pm^{\mu,\rho}\subset\fC_\pm^{\mu,\rho}$ be the subspace of those generalized functions from 
$\fC_\pm^{\mu,\rho}$ whose support is contained in $\im (K\to G(\bA)/T(F)N(\bA))$. Note that $\dim V_\pm^{\mu,\rho}<\infty$.

\begin{lem}   \label{l:things are simple}
(i) The map $\fR^\mu\otimes V_\pm^{\mu,\rho}\to\fC_\pm^{\mu,\rho}$ is an isomorphism.

(ii) The map $R^\mu\otimes V_\pm^{\mu,\rho}\to\C_\pm^{\mu,\rho}$ is an isomorphism.

(iii) The map $R^\mu\otimes_{\fR^\mu} \fC_\pm^{\mu,\rho}\to\C_\pm^{\mu,\rho}$ is an isomorphism.

(iv) The isomorphisms (i)-(iii) preserve the filtrations.
\end{lem}

\begin{proof}
It suffices to prove the statements about the maps $\fR^\mu\otimes V_\pm^{\mu,\rho}\to\fC_\pm^{\mu,\rho}$ and 
$R^\mu\otimes V_\pm^{\mu,\rho}\to\C_\pm^{\mu,\rho}$. They
follow from the decomposition $G(\bA)=K\cdot N(\bA)\cdot T(\bA)$. 
\end{proof}

\subsection{The statements to be proved}
From now on we fix both $\mu$ and $\rho$.
The operator $M:\fC_+\to\fC_-$ induces $\fR^\mu$-module homomorphisms $\fC^{\mu,\rho}_+\to\fC^{\mu,\rho}_-$ and $\C^{\mu,\rho}_+\to\C^{\mu ,\rho}_-$, which will still be denoted by $M$. One has $M(\fC^{\mu,\rho}_{\ge Q })\subset\fC^{\mu,\rho}_{\le -Q}\,$.

\begin{prop}   \label{p:2invertibility}
(i) The operators $M:\fC^{\mu,\rho}_+\to\fC^{\mu,\rho}_-$ and $M:\C^{\mu,\rho}_+\to\C^{\mu ,\rho}_-$ are invertible.

(ii) $M^{-1} (\fC^{\mu,\rho}_{\le -Q})\subset\fC^{\mu,\rho}_{\ge Q-a(\rho )}\,$, where $a(\rho)$ depends only on the non-Archimedean local components of $\rho$. If each non-Archimedean local component of $\rho$ is trivial then one can take $a(\rho)=0$.
\end{prop}

We will deduce the proposition from the corresponding local statements, see Subsections~\ref{ss:local statements}-\ref{ss:itog} below.

\begin{rem}  \label{r:continuity_automatic}
$\fC^{\mu,\rho}_\pm$, $\C^{\mu,\rho}_\pm$   are topological vector spaces and $\fR^\mu$, $R^\mu$ are topological algebras. 
 The isomorphisms from \lemref{l:things are simple} are topological. So \propref{p:2invertibility}(i) implies that $M:\fC^{\mu,\rho}_+\to\fC^{\mu,\rho}_-$ and $M:\C^{\mu,\rho}_+\to\C^{\mu ,\rho}_-$ are \emph{topological} isomorphisms.
\end{rem}

\subsection{Local statements}  \label{ss:local statements}
Let $v$ be a place of $F$. Let $\mu_v$ and $\rho_v$ denote the $v$-components of $\mu$ and $\rho$.

For $x\in F_v^\times$ we set $\deg x:=-\log_A |x|$, where $A$ is the number that we fixed in 
Subsection~\ref{sss:the number a}. Let 
$$\deg :K_v\backslash G(F_v)/N(F_v)\to\bR$$ denote the map that takes $\diag  (x,x^{-1})$ to 
$\deg x$.

\subsubsection{The spaces $\C_{\pm,v}^{\mu,\rho}$ and $\fC_{\pm ,v}^{\mu,\rho}\,$}
For any $Q\in\bR$, let $\fC_{\ge Q,v}$ denote the space of $K_v$-finite generalized functions
$f$ on $G(F_v)/N(F_v)$ such that
\begin{equation*}  
\Supp (f)\subset \{ x\in G(F_v)/N(F_v)\,|\,\deg x\ge Q\}.
\end{equation*}
Let $\C_{\ge Q, v}\subset\fC_{\ge Q, v}$ denote the subspace of smooth functions; if $v$ is non-Archimedean then $\C_{\ge Q,v}=\fC_{\ge Q,v}\,$.
Let $\fC_{+, v}$ denote the union of  $\fC_{\ge Q, v}$ for all $Q\in\bR$. Similarly, we have the spaces $\fC_{\le Q, v}\,$, $\fC_{-, v}\,$, $\C_{\le Q, v}\,$,
$\C_{\pm, v}\,$.

Define the action of $F_v^\times$ on $\C_{\pm, v}$ similarly to Subsect.~\ref{sss:change}.
Let $\C_{+,v}^\mu\subset\C_{+,v}$ denote the $\mu_v$-eigenspace for the $O_v^\times$-action and let $\C_{+,v}^{\mu ,\rho}\subset\C_{+,v}^\mu$ denote the maximal subspace on which $K_v$ acts according to $\rho_v\,$.
Similarly, we have  $\C_{-,v}^{\mu ,\rho}$ and $\fC_{\pm ,v}^{\mu ,\rho}\,$.

\subsubsection{The algebra $\fR^\mu_v$ and its action on $\fC^{\mu,\rho}_{\pm ,v}$}   \label{sss:local ring R}
Let $\fR_{\ge Q,v}$ denote the space of distributions $\eta$ on $F_v^\times$ such that $\deg x\ge Q$ for all $x\in\Supp\eta$.
Let $\fR_v$ denote the union of $\fR_{\ge Q,v}$ for all $Q\in\bR$. Then $\fR_v$ is a filtered algebra with respect to convolution.

Let $e_v^\mu\in\fR_v$ be the product of $\mu_v^{-1}$ and the normalized Haar measure on the compact subgroup 
$O_v^\times\subset F_v^\times$. Clearly $e_v^\mu$ is an idempotent.

Set $\fR_v^\mu:=e_v^\mu\cdot\fR_v\,$; this is a unital algebra. 
Let $R_v^\mu\subset\fR_v^\mu$ denote the ideal formed by smooth measures. If $v$ is non-Archimedean  then $R_v^\mu =\fR_v^\mu\,$. 
One has
\[
 \fR^\mu_{\ge Q,v}\cdot\fC^\mu_{\ge Q',v}\subset\fC^\mu_{\ge Q'+Q,v}\, , \quad \fR^\mu_{\ge Q,v}\cdot\fC^\mu_{\le Q',v}\subset\fC^\mu_{\le Q'-Q,v}\, . 
\]

Let $V_{\pm,v}^{\mu,\rho}\subset\fC_\pm^{\mu,\rho}$ be the subspace of those generalized functions from $\fC_{\pm,v}^{\mu,\rho}$ whose support is contained in $\im (K_v\to G(F_v)/N(F_v))$. Note that $\dim V_{\pm,v}^{\mu,\rho}<\infty$.

\begin{lem}   \label{l:2things are simple}
(i) The map $\fR_v^\mu\otimes V_{\pm ,v}^{\mu,\rho}\to\fC_{\pm ,v}^{\mu,\rho}$ is an isomorphism.

(ii) The map $R_v^\mu\otimes V_{\pm ,v}^{\mu,\rho}\to\C_{\pm ,v}^{\mu,\rho}$ is an isomorphism.

(iii) The map $R_v^\mu\otimes_{\fR_v^\mu} \fC_{\pm ,v}^{\mu,\rho}\to\C_{\pm ,v}^{\mu,\rho}$ is an isomorphism.

(iv) The isomorphisms (i)-(iii) preserve the filtrations. \qed
\end{lem}

\subsubsection{The operator $M_v\,$}
Fix a Haar measure on $F_v=N(F_v)$. 
Then one defines an $\fR_v^\mu$-module homomorphism $M_v:\fC_{+ ,v}^{\mu,\rho}\to\fC_{- ,v}^{\mu,\rho}$ by
\[
(M_v\varphi )(x)=\int\limits_{N(F_v)}\varphi (xnw)dn, \quad\quad  \varphi\in\fC_{+ ,v}^{\mu,\rho}\, ,\,  x\in G(F_v)/N(F_v ),
\]
where $w:=\left(\begin{matrix} 0&1\\-1&0\end{matrix}\right)\in SL(2)$. One has $M_v(\fC^{\mu,\rho}_{\ge Q, v })\subset\fC^{\mu,\rho}_{\le -Q,v}\,$, 
$M_v(\C_{+ ,v}^{\mu,\rho})\subset\C_{- ,v}^{\mu,\rho}\,$. 

\begin{rem}
One has $G(F_v)/N(F_v )=F_v^2 -\{0\}$, and the operator $M_v$ is essentially the 2-dimensional Radon transform. However, the functional spaces 
$\fC_{\pm ,v}$ and $\C_{\pm ,v}$ are not standard for the theory of Radon transform. 
\end{rem}

\begin{prop}   \label{p:local invertibility}
(i) The operators $M_v:\fC_{+ ,v}^{\mu,\rho}\to\fC_{- ,v}^{\mu,\rho}$ and $M_v:\C_{+ ,v}^{\mu,\rho}\to\C_{- ,v}^{\mu,\rho}$ are invertible.

(ii) There exists a number $a=a(\rho_v)$ such that 
$$M_v^{-1} (\fC^{\mu,\rho}_{\le -Q,v})\subset\fC^{\mu,\rho}_{\ge Q-a,v}\, ,\quad M_v^{-1} (\C^{\mu,\rho}_{\le -Q,v})\subset\C^{\mu,\rho}_{\ge Q-a,v}\, .$$
If $v$ is Archimedean one can take $a=0$.
\end{prop}

The statements about the operator $M_v:\C_{+ ,v}^{\mu,\rho}\to\C_{- ,v}^{\mu,\rho}$ are proved in \cite{W}. In the Archimedean case they are proved in \cite{W} using invertibility of certain elements of the algebra $\fR_{\ge 0,v}\,$; the same argument works if $\C_{\pm ,v}^{\mu,\rho}$ is replaced by $\fC_{\pm ,v}^{\mu,\rho}\,$. In the non-Archimedean case there is no difference between $\C_{\pm ,v}^{\mu,\rho}$ and $\fC_{\pm ,v}^{\mu,\rho}\,$.

\begin{rem}
The work \cite{W} contains explicit formulas for the operator $M_v^{-1}$. 
\end{rem}

\subsection{$\fC_\pm^{\mu,\rho}$ and $M$ as infinite tensor products}   \label{ss:infinite tensor}
Let $v$ be a non-Archimedean place of $F$ such that $\rho_v$ is the unit representation and $\mu_v$ is trivial. Then
$\fC_{\pm ,v}^{\mu,\rho}$ contains a canonical element, namely the function on $G(F_v)/N(F_v)$ that equals $1$ on the image of $K_v$ and equals $0$ elsewhere. This element of $\fC_{\pm ,v}^{\mu,\rho}$ will be denoted by $\delta_{\pm ,v}\,$.

Given a collection of elements $f_v\in\fC_{\pm ,v}^{\mu,\rho}$ such that $f_v=\delta_{\pm ,v}$ for almost all $v$, one can form the (generalized) function $\underset{v}{\otimes} f_v$  on $G(\bA )/N(\bA )$. Then the pushforward of $\underset{v}{\otimes} f_v$ to $G(\bA )/T(F)N(\bA )$ is an element of $\fC_{\pm}^{\mu,\rho}$. Thus one gets an $\fR^\mu$-linear map
\begin{equation}   \label{e:infinite tens}
\underset{v}{\otimes}\widetilde\fC_{\pm ,v}^{\mu,\rho}\to\fC_{\pm}^{\mu,\rho}\, ,
\end{equation}
where
$\widetilde\fC_{\pm ,v}^{\mu,\rho}:=\fR^\mu\otimes_{\fR_v^\mu}\fC_{\pm ,v}^{\mu,\rho}\,$ and the symbol $\underset{v}{\otimes}$ in \eqref{e:infinite tens} is understood as restricted tensor product over the ring $\fR^\mu$.

\begin{lem}   \label{l:inftens}
The map \eqref{e:infinite tens} is an isomorphism. It preserves the filtrations.
\end{lem}

\begin{proof}
Use Lemmas~\ref{l:things are simple}(i,iv) and \ref{l:2things are simple}(i,iv).
\end{proof}

Now let us explain the relation between the operator $M:\fC_{+}^{\mu,\rho}\to\fC_{-}^{\mu,\rho}$ and the local operators 
$M_v:\fC_{+ ,v}^{\mu,\rho}\to\fC_{- ,v}^{\mu,\rho}\,$.

The operators $M_v$ depend on the choice of Haar measures on the local fields $F_v\,$. Choose them so that $\mes O_v=1$ for almost all $v$ and the product measure on $\bA$ satisfies the condition $\mes (\bA/F)=1$.

As before, set $\widetilde\fC_{\pm ,v}^{\mu,\rho}:=\fR^\mu\otimes_{\fR_v^\mu}\fC_{\pm ,v}^{\mu,\rho}\,$. The $\fR_v^\mu$-linear map
$M_v:\fC_{+ ,v}^{\mu,\rho}\to\fC_{- ,v}^{\mu,\rho}$ induces an $\fR^\mu$-linear map $\widetilde M_v:\widetilde\fC_{+,v}^{\mu,\rho}\to\widetilde\fC_{-,v}^{\mu,\rho}\,$.

For almost all $v$ one has the elements $\delta_{\pm ,v}\in\fC_{\pm ,v}^{\mu,\rho}\,$, and one has  
$\widetilde M_v(\delta_{+,v})=\alpha_v\cdot\delta_{-,v}$ for some $\alpha_v\in\fR^\mu$. It is easy to see that the elements $\alpha_v$ converge to $1$ with respect to the filtration formed by $\fR^\mu_{\le Q}\,$. So one can form the infinite tensor product
\begin{equation}  \label{e:2infinite tens}
\underset{v}{\otimes}\widetilde M_v\, ,
\end{equation}
which is an operator $\underset{v}{\otimes}\widetilde\fC_{+,v}^{\mu,\rho}\to \underset{v}{\otimes}\widetilde\fC_{-,v}^{\mu,\rho}\,$.

\begin{lem}  \label{l:2inftens}
The isomorphism \eqref{e:infinite tens} identifies $M:\fC_{+}^{\mu,\rho}\to\fC_{-}^{\mu,\rho}$ with the operator~\eqref{e:2infinite tens} . \qed
\end{lem}

\subsection{Proof of \propref{p:2invertibility}} \label{ss:itog}
The statements of \propref{p:2invertibility} about the operator 
$M:\fC_{+}^{\mu,\rho}\to\fC_{-}^{\mu,\rho}$ and its inverse follow from \propref{p:local invertibility} combined with Lemmas~\ref{l:inftens}-\ref{l:2inftens}. 
By \lemref{l:things are simple}(iii), invertibility of $M:\C_{+}^{\mu,\rho}\to\C_{-}^{\mu,\rho}$ follows from the  invertibility of 
$M:\fC_{+}^{\mu,\rho}\to\fC_{-}^{\mu,\rho}$. \qed

\appendix

\section{Relation to works on the geometric Langlands program}   \label{s:miraculous}
In this appendix we relate this article to \cite{DG2,DG3, G1}. Subsections~\ref{sss:CT!}-\ref{sss:why}, \ref{ss:psid}-\ref {ss:miraculous}, and \ref{sss:Eis!}-\ref{sss:Dennis-iso}  are the nontrivial ones. 

In Subsections~\ref{sss:CT!}-\ref{sss:why} we motivate the definition of the subspace $\cA_{ps-c}\subset\cA$ given in Subsection~\ref{ss:Aps-c}. In Subsections~ \ref{ss:psid}-\ref {ss:miraculous} we motivate the definition of the function $b$ from Subsection~\ref{sss:b} and the definition of the form $\B$. In Subsections~ \ref{sss:Eis!}-\ref{sss:Dennis-iso}  we discuss the relation between the operator $\Eis'$ from Subsection~ \ref{ss:second Eis} and the functor $\Eis_!$ from \cite{DG3}.

\subsection{D-modules, $l$-adic sheaves, and functions}  \label{ss:D&ell}
\subsubsection{`Left' and `right' functors}   \label{sss:left&right}
We will consider two different cohomological formalisms:

\smallskip

(i) Constructible $l$-adic sheaves on schemes of finite type over a field;

(ii) D-modules on schemes of finite type over a field of characteristic~0.

\smallskip

In each of them we have two adjoint pairs of functors $(f^*, f_*)$ and $(f_!, f^!)$. We will refer to $f^*$ and $f_!$ as `left' functors and to 
$f_*$ and $f^!$ as `right' functors (each `left' functor is left adjoint to the `right' functor from the same pair). Caveat: in the D-module setting the `left' functors are, in general, only partially defined (because D-modules are not assumed holonomic). Thus in the D-module setting we \emph{have} to consider the `right' functors as the `main' ones. We prefer to do this in the constructible setting as well (then the analogy between the two settings becomes transparent).

Both cohomological formalisms (i) and (ii) exist in the more general setting of \emph{algebraic stacks locally of finite type} over a field\footnote{See \cite{DG1} for the D-module formalism and \cite{Beh1, Beh2, LO,GLu} for the $l$-adic one.}, but the situation with the pushforward functor is subtle (in the D-module setting it is discussed in Subsect.~\ref{ss:reform} below). However, if a morphism $f$ between algebraic stacks is \emph{representable}\footnote{Representability means that the fibers of $f$ are algebraic \emph{spaces}.} and has \emph{finite} type\footnote{In fact, the combination \{representability\}+\{finite type\} can be replaced by a weaker condition of \emph{safety}. The definition of safety is contained in Subsect.~\ref{ss:reform}.} then $f_*$ and $f_!$ are as good as in the case of schemes.
 
\subsubsection{Functions-sheaves dictionary (an unusual convention)}   \label{sss:Functions-sheaves}
Let $\cY$ be an algebraic stack locally of finite type over $\bF_q\,$ and let $M$ be an object of the bounded constructible 
derived category  of $\ql$-sheaves on $\cY$. To such a pair we associate a `trace function' on the groupoid\footnote{A function on a groupoid is the same as a function on the set of its isomorphism classes, but the notion of direct image of a function is slightly different.} $\cY (\bF_q)$. We do it in an unconventional way: namely, \emph{our trace function corresponding to $M$ equals Grothendieck's trace function corresponding to the Verdier dual $\bD M$} (in other words, the value of our trace function at $y\in\cY (\bF_q)$ is the trace of the \emph{arithmetic} Frobenius acting on the \emph{!-stalk} of $M$ at $y$).

Thus the pullback of functions corresponds to the $!$-pullback of $l$-adic complexes, and the pushforward of functions with respect to a morphism $f$ of finite type corresponds\footnote{If $f$ is not representable (or safe) then explaining the precise meaning of the word `corresponds' requires some care
(see  \cite{Beh1, Beh2, Sun,GLu}) because the pushforward of a bounded complex is not necessarily bounded.}
 to the $*$-pushforward of $l$-adic complexes. In other words, the standard operators between spaces of functions correspond to the `right' functors in the sense of Subsect.~\ref{sss:left&right}.

\begin{ex}
According to our convention, the constant function 1 corresponds to the \emph{dualizing complex} of $\cY$, which will be denoted by $\dK_{\cY}\,$.
\end{ex}

\begin{ex}   \label{ex:b}
In Subsect.~\ref{sss:b} we defined a function $b$ on $(\Bun_G\times\Bun_G)(\bF_q)$, where $G:=SL(2)$. According to our new convention, $b$ corresponds to the complex $\Delta_!(\dK_{\Bun_G})$, where $\Delta :\Bun_G\to\Bun_G\times\Bun_G$ is the diagonal. The stack $\Bun_G$ is smooth of pure dimension $d=3g_X-3$, where $g_X$ is the genus of the curve $X$. So the function $q^{-d}\cdot b$ corresponds to 
$\Delta_!((\ql)_{\Bun_G})$.
\end{ex}

\begin{rem}    \label{r:Tate twist}
Let $M$ be an $l$-adic complex on $\cY$ and $f$ the corresponding function on $\cY (\bF_q )$. According to our convention, the function corresponding to $M[2](1)$ equals $qf$.
\end{rem}

\subsubsection{Analogy between D-modules and functions} \label{sss:Dmods&functions}
For certain reasons (including serious ones) the works \cite{DG2,DG3, G1} deal with D-modules but not with constructible sheaves. There is no direct relation between D-modules (which live in characteristic 0) and functions on $\cY (\bF_q)$, where $\cY$ is as in Subsect.~\ref{sss:Functions-sheaves}. However there is an analogy between them. It comes from the analogy between the two cohomological formalisms considered in Subsect.~\ref{sss:left&right} and the functions-sheaves dictionary as formulated in Subsect.~\ref{sss:Functions-sheaves}.

\subsection{Some categories of D-modules on $\Bun_G$ and $\Bun_T\,$} 
\subsubsection{} \label{sss:stack convention}
Let $k$ denote an algebraically closed field  of \emph{characteristic zero} and $X$ a smooth complete connected curve over $k$. 
Just as in the rest of the article, $G:=SL(2)$ and $T\subset G$ is the group of diagonal matrices.\footnote{In \cite{DG2,DG3, G1} instead of 
$SL(2)$ one considers any reductive group, and instead of $T$ one considers the Levi quotient of any parabolic.} Let $\Bun_G$ (resp.~ $\Bun_T$) denote the moduli stack of principal $G$-bundles (resp.~$T$-bundles) on $X$.

We will say `stack' instead of `algebraic stack locally of finite type over $k$ whose $k$-points have affine automorphism groups'. We will mostly deal with the stacks $\Bun_G$ and $\Bun_T\,$, which are not quasi-compact.

\subsubsection{} 
For any stack $\CY$ over $k$ one has the DG category of (complexes of) 
D-modules on $\CY$, denoted by $\Dmod(\CY)$. Let $\Dmod(\CY)_c\subset\Dmod(\CY)$ denote the full subcategory of objects $M\in\Dmod(\CY)$ such that for some quasi-compact open $U\overset{j}\mono\cY$ the morphism $M\to j_*j^*M$ is an isomorphism. Let $\Dmod(\CY)_{ps-c}\subset\Dmod(\CY)$ denote the full subcategory of objects $M\in\Dmod(\CY)$ such that for some quasi-compact open $U\overset{j}\mono\cY$ the object $ j_!j^!M$ is defined and the morphism $ j_!j^!M\to M$ is an isomorphism.

\begin{rem}   \label{r: truncatability}
In D-module theory the functor $j_!$ is only partially defined, in general. An open quasi-compact substack $U\overset{j}\mono\cY$ is said to be \emph{co-truncative} if $j_!$ is defined \emph{everywhere.}\footnote{Typical example: if $V$ is a finite-dimensional vector space then 
$(V-\{ 0\})/\bG_m$ is a co-truncative substack of $V/\bG_m\,$.} A stack $\cY$ is said to be \emph{truncatable} if every quasi-compact open substack of $\cY$ is contained in a co-truncative one. The stacks $\Bun_G$ and $\Bun_T$ are truncatable. For $\Bun_T$ this is obvious (because each connected component of $\Bun_T$ is quasi-compact); for $\Bun_G$ this is proved in \cite{DG2}.
 \end{rem}
 
 \subsubsection{} 
 Note that $\Dmod(\Bun_T)_c=\Dmod(\Bun_T)_{ps-c}$ (because each connected component of $\Bun_T$ is quasi-compact). On the other hand, 
$$\Dmod(\Bun_G)_c\ne\Dmod(\Bun_G)_{ps-c}\, .$$

\begin{rem}
The approach of \cite{DG1,DG2,DG3,G1} is to work only with \emph{cocomplete} DG categories (i.e., those in which arbitrary inductive limits are representable). The DG category $\Dmod (\cY )$ is cocomplete for any stack $\cY$. On the other hand, $\Dmod (\cY )_c$ and 
$\Dmod (\cY )_{ps-c}$ are not cocomplete if $\cY$ equals $\Bun_G$ or $\Bun_T\,$.
\end{rem}

The reader may prefer to skip the next remark.

\begin{rem}   \label{r:skip}
For any cocomplete DG category $\cD$, let $\cD'\subset\cD$ denote the following full subcategory: $M\in\cD'$ if and only if there exists a finite collection $S$ of compact objects of $\cD$ such that $M$ belongs to the cocomplete DG subcategory of $\cD$ generated by $S$. For any truncatable stack $\cY$, one has the following description of $\Dmod (\cY )_{ps-c}$ and $\Dmod (\cY )_c$  in terms of the DG category $\Dmod (\cY )$ and its Lurie dual 
$\Dmod (\cY )^\vee$ (the latter two DG categories are cocomplete):
\begin{equation}  \label{e:Dps-c}
\Dmod (\cY )_{ps-c}=\Dmod (\cY )' ,
\end{equation}
\begin{equation}    \label{e:Dc}
 \Dmod (\cY )_c=(\Dmod (\cY )^\vee)'.
\end{equation}
To prove \eqref{e:Dps-c}, use \cite[Prop.~2.3.7]{DG2} and the following fact: for any quasi-compact\footnote{According to the convention of Subsect.~\ref{sss:stack convention}, stacks are assumed to be \emph{locally} of finite type. So quasi-compactness is the same as having finite type.} stack $\cZ$, the DG category $\Dmod (\cZ )$ is generated by finitely many compact objects.\footnote{Without finiteness, this is \cite[Thm.~8.1.1]{DG1}. To prove the finiteness statement, use a stratification argument combined with \cite[Lemmas~10.3.6 and 10.3.9]{DG1} and \cite[ Cor.~8.3.4]{DG1} to reduce to the case where $\cZ$ is  a smooth affine scheme and the case $\cZ=(\Spec k)/G$, where $G$ is an  algebraic group. The first case is clear. In the second case $\Dmod (\cZ )$ is generated by the $!$-direct image of $k\in\Vect=\Dmod (\Spec k)\,$, which is a compact object of $\Dmod (\cZ )$.} To prove  \eqref{e:Dc}, one can use the description of $\Dmod (\cY )^\vee$ given in  \cite[Cor.~4.3.2]{DG2} or  \cite[Subsect.~1.2]{G1}.  
\end{rem}

\subsection{D-module analogs of $\cA^K$, $\cA_c^K$,  $\C^K$, and $\C_c^K$} \label{ss:easy analogs}
Let $F$ be a function field. Then the space $\cA^K$ (i.e., the subspace of $K$-invariants in $\cA$) identifies with the space of all functions on
$K\backslash G(\bA )/G(F)$, i.e., on the set of isomorphism classes of $G$-bundles on the smooth projective curve over $\bF_q$ corresponding to $F$. So we consider the DG category $\Dmod(\Bun_G)$ to be an analog of the vector space $\cA^K$. This DG category was studied in  \cite{DG2,DG3, G1} in the spirit of `geometric functional analysis' (with complexes of D-modules playing the role of functions and cocomplete DG categories playing the role of abstract topological vector spaces). Let us note that the D-module analog of the whole space~$\cA$ has not been studied in this spirit, and it is not clear how to do it.

Because of the convention of Subsect.~\ref{sss:Functions-sheaves}, we consider
$\Dmod(\Bun_G)_c$ to be an analog of $\cA_c^K$. 

We consider $\Dmod(\Bun_T)$ to be a D-module analog of the space $\C^K$ (the reason is clear from Example~\ref{ex:functional}).
We consider $\Dmod(\Bun_T)_c$ to be an analog of $\C_c^K$. Similarly to the subspaces $\C_{\pm}^K\subset\C$ (see Subsect.~\ref{ss:C}) one defines the full  subcategories $\Dmod(\Bun_T)_{\pm}\subset\Dmod(\Bun_T)$. 

There is also a (non-obvious) analogy between $\Dmod(\Bun_G)_{ps-c}$ and $\cA_{ps-c}^K\,$. It will be explained in Subsect.~\ref{ss:CTAps-c} below. But first we have to recall some material from~\cite{DG1}.

\subsection{Good and bad direct image for D-modules}   \label{ss:reform}
For any morphism $f:\cY'\to\cY$ between quasi-compact algebraic stacks one defines in \cite{DG1} \emph{two} pushforward functors\footnote{See \cite[Sect. 0.5.9]{DG1}, \cite[Sects. 7.4-7.8]{DG1}, and  \cite[Sect. 9]{DG1}. The particular case where $\cY=\Spec k$ and $\cY'$ is the classifying stack of an algebraic group is discussed in  \cite[Sect.~ 7.2]{DG1} and \cite[Example 9.1.6]{DG1}. 
Let us note that instead of $f_*$ one uses in \cite{DG1} the more precise notation $f_{\on{dR},*}\,$, where dR stands for `de Rham'.}: the `usual' functor $f_*:\Dmod (\cY')\to \Dmod (\cY)$ (which is very dangerous, maybe pathological) and the \emph{`renormalized direct image'} $f_\blacktriangle:\Dmod (\cY')\to \Dmod (\cY)$ (which is nice). One also defines a canonical morphism $f_\blacktriangle\to f_*\,$, which is an isomorphism if and only  if $f$ is \emph{safe}. By definition, a quasi-compact morphism $f$ is \emph{safe} if for any geometric point $y\to\cY$ and any geometric point $\xi: y'\to \cY'_y:=\cY'\underset{\cY}\times y$ the neutral connected component of the automorphism group of $\xi$ is unipotent.
For instance, any representable morphism is safe. We will use the functor $f_*$ only for safe morphisms $f$ (in which case 
$f_*=f_\blacktriangle$).

The nice properties of $f_\blacktriangle$ are continuity\ (i.e., commutation with infinite direct sums) and base change with respect to !-pullbacks.
Because of base change, we consider $f_\blacktriangle$ as a `right' functor (in the sense of Subsect.~\ref{sss:left&right}), even though if $f$ is not safe then $f_\blacktriangle$ is not right adjoint to the partially defined functor $f^*$.

Base change allows to define $f_\blacktriangle$ if $f$ is quasi-compact while $\cY$ is not.
Moreover, one defines $f_\blacktriangle (M)$ if $f$ is not necessarily quasi-compact but $M\in\Dmod (\cY')$ is such that $M=j_*j^*M$ for some open substack 
$U\overset{j}\mono\cY'$ quasi-compact over $\cY$: namely, one sets $f_\blacktriangle (M):=(f\circ j)_\blacktriangle (j^*M)$ for any $U$ with the above property.

Finally, if $\cY=\Spec k$ and $M\in\Dmod (\cY')$ then one writes $\Gamma_{\on{ren}} (\cY', M)$ instead of $f_\blacktriangle (M)$. The functor 
$\Gamma_{\on{ren}}$ is called \emph{renormalized de Rham cohomology}.

\subsection{D-module analogs of $\cA_{ps-c}^K$ and $\CT^K$} \label{ss:CTAps-c}
We will consider $\Dmod(\Bun_G)_{ps-c}$ to be an analog of $\cA_{ps-c}^K\,$. The goal of this subsection is to justify this.

Recall that the subspace $\cA_{ps-c}\subset\cA$ is defined in Subsect.~\ref{ss:Aps-c} in terms of the operator $\CT:\cA\to\C\,$. So the first step is to define a D-module analog of the corresponding operator $\CT^K:\cA^K\to\C^K$. We will do this in Subsect.~\ref{sss:CT*} using the diagram 
\begin{equation} \label{e:basic diagram intro}
\xy
(-15,0)*+{\Bun_G}="X";
(15,0)*+{\Bun_T}="Y";
(0,15)*+{\Bun_B}="Z";
{\ar@{->}_{\sfp} "Z";"X"};
{\ar@{->}^{\sfq} "Z";"Y"};
\endxy
\end{equation}
that comes from the diagram of groups $G\hookleftarrow B\twoheadrightarrow T$ (as usual, $B\subset G$ is the subgroup of upper-triangular matrices). 

\begin{rem}   \label{r:the_2_diagrams}
Diagram~\eqref{e:basic diagram intro} is closely related to diagram~\eqref{e:1pull-push}. The relation is as follows. Suppose for a moment that $X$ is a curve over $\bF_q$ (rather than over a field of characteristic 0). Then the quotient of diagram~\eqref{e:1pull-push} by the action  of the maximal compact subgroup $K\subset G(\bA )$ identifies with the diagram
\[
\Bun_G(\bF_q)\leftarrow\Bun_B(\bF_q)\to\Bun_T(\bF_q)
\]
corresponding to \eqref{e:basic diagram intro}.
\end{rem}

\begin{rem} \label{r:safety}
Unipotence of $\Ker (B\epi T)$ easily implies that the morphism $$\sfq:\Bun_B\to\Bun_T$$ is safe in the sense of Subsect.~\ref{ss:reform} 
(although $\sfq$ is not representable).
\end{rem}

\subsubsection{The functor $\CT_*$ as a D-module analog of the operator $\CT^K$}  \label{sss:CT*}
Following \cite{DG3}, consider the functor 
\[
\CT_*:\Dmod (\Bun_G)\to\Dmod (\Bun_T),\quad \CT_*:=\sfq_*\circ\sfp^!\, .
\]
 Note that by \remref{r:safety} and Subsect.~\ref{ss:reform}, the functor $\sfq_*$ equals $\sfq_\blacktriangle\,$, so it is not pathological.

Recall that the operator $\CT:\cA\to\C$ is the pull-push along diagram~\eqref{e:1pull-push} (see Subsect.~\ref{ss:CT}). So \remref{r:the_2_diagrams} allows us to consider the functor $\CT_*$ as a D-module analog\footnote{The operator $\CT^K$ also has another (more refined) D-module analog, namely the functor $\CT^{\enh}=\CT_B^{\enh}$ discussed in Subsection~\ref{ss:enh}.} of the operator $\CT^K$.

\subsubsection{The functor $\CT_!$ and its relation to $\CT_*\,$}  \label{sss:CT!}
In \cite{DG3} one defines another functor 
\[
\CT_!:\Dmod (\Bun_G)\to\Dmod (\Bun_T)
\]
by the formula 
\begin{equation} \label{e:!?}
\CT_!:=\sfq_!\circ\sfp^*\, ,
\end{equation}
which has to be understood in a \emph{subtle} sense. The subtlety is due to the fact that the r.h.s. of \eqref{e:!?} involves `left' functors. Because of that, the r.h.s. of \eqref{e:!?} is, \emph{a priori,} a functor $\Dmod (\Bun_G)\to\on{Pro} (\Dmod (\Bun_T))$, where `Pro' stands for the DG category of pro-objects.
However, the main theorem of \cite{DG3} says that the essential image of this functor is contained in 
$\Dmod (\Bun_T)\subset \on{Pro} (\Dmod (\Bun_T))$. It also says that one has a canonical isomorphism
\begin{equation}  \label{e:!!}
\CT_!\simeq\iota^*\circ\CT_*\, ,
\end{equation}
where $\iota^*:\Dmod (\Bun_T)\iso\Dmod (\Bun_T)$ is the pullback along the inversion map $\iota:\Bun_T\iso\Bun_T$ .

\subsubsection{Why  $\Dmod(\Bun_G)_{ps-c}$ is an analog of $\cA_{ps-c}^K\,$}  \label{sss:why}
By \lemref{l:Acirc}, 
\[
\cA_c^K=\{ f\in\cA^K\,|\, \CT (f)\in\Cm^K\}.
\] 
A similar easy argument shows that
\[
\Dmod (\Bun_G)_c=\{ \cF\in\Dmod (\Bun_G)\,|\, \CT_* (\cF )\in\Dmod (\Bun_T)_-\},
\] 
\begin{equation}   \label{e:ps-c via CT!}
\Dmod (\Bun_G)_{ps-c}=\{ \cF\in\Dmod (\Bun_G)\,|\, \CT_! (\cF )\in\Dmod (\Bun_T)_-\}.
\end{equation}
Now combining \eqref{e:!!} and  \eqref{e:ps-c via CT!}, we see that
\begin{equation}   \label{e:2ps-c via CT!}
\Dmod (\Bun_G)_{ps-c}=\{ \cF\in\Dmod (\Bun_G)\,|\, \CT_* (\cF )\in\Dmod (\Bun_T)_+\}.
\end{equation}
Formula \eqref{e:2ps-c via CT!} makes clear the analogy between $\Dmod (\Bun_G)_{ps-c}$ and the space
\[
\cA_{ps-c}^K:=\{ f\in\cA^K\,|\, \CT (f)\in\Cp^K\}
\]
introduced in Subsection~\ref{ss:Aps-c}.

\subsection{D-module analog of $E$} \label{sss:Vect}
In Sect.~\ref{sss:E} we fixed a field $E$; according to our convention, all functions take values in $E$. Thus $E$ is the space of functions on a point.

So the D-module analog of $E$ is the DG category $\Vect:=\Dmod (\Spec k)$, which is just the DG category of complexes of vector spaces over 
$k$.

\subsection{D-module analog of $\B_{naive}^K\,$} 
Recall that $\B_{naive}$ denotes the usual pairing between $\cA_c$ and $\cA$. Let $\B_{naive}^K$ denote the restriction of $\B_{naive}$ to $K$-invariant functions.

In Subsect.~\ref{sss:Vect} we defined the DG category $\Vect$. The D-module analog of $\B_{naive}^K$ is the functor
\[
\Dmod (\Bun_G)_c\times\Dmod (\Bun_G)\to\Vect, \quad\quad (M_1,M_2)\mapsto\Gamma_{\on{ren}}  (M_1\otimes M_2).
\]
Here $\Gamma_{\on{ren}}$ is the renormalized de Rham cohomology (see Subsect.~\ref{ss:reform})
and $\otimes$ stands for the $!$-tensor product, i.e.,
$M_1\otimes M_2:=\Delta^! (M_1\boxtimes M_2)$, where $\Delta:\Bun_G\times\Bun_G\to\Bun_G$ is the diagonal.

\subsection{D-module analogs of  $\B^K\,$ and $L^K$} \label{ss:psid}
\subsubsection{The pseudo-identity functor}  \label{sss:ps-id}
Let $\cY$ be a stack. Let $\on{pr}_1,\on{pr}_2:\cY\times\cY\to\cY$ denote the projections and $\Delta:\cY\to\cY\times\cY$ the diagonal morphism.
Any $\cF\in\Dmod (\cY\times\cY)$ defines functors
\begin{equation} \label{e:functor-cF}
\Dmod (\cY)_c\to\Dmod (\cY),\quad M\mapsto (\on{pr}_1)_\blacktriangle (\cF\otimes \on{pr}_2^!M),
\end{equation}
\begin{equation}   \label{e:pairing-cF}
\Dmod (\cY)_c\times\Dmod (\cY)_c\to \Vect ,\quad (M_1,M_2)\mapsto\Gamma_{\on{ren}} (\cY\times\cY ,\cF\otimes (M_1\boxtimes M_2)),
\end{equation}
where $ (\on{pr}_1)_\blacktriangle$ is the renormalized direct image and $\Gamma_{\on{ren}}$ is the renormalized de Rham cohomology (see Subsect.~\ref{ss:reform}).

For example, if $\cF=\Delta_*\omega_{\cY}$ then \eqref{e:functor-cF} is the identity functor and the `pairing' \eqref{e:pairing-cF} takes $(M_1,M_2)$ to $\Gamma_{\on{ren}} (\cY ,M_1\otimes M_2)$.

Now let $k_{\cY}$ denote the Verdier dual $\bD\omega_{\cY}$ (a.k.a. the constant sheaf\footnote{If $k=\bC$ then the Riemann-Hilbert correspondence takes $\omega_{\cY}$ to the dualizing complex and $k_{\cY}$ to the constant sheaf.}). Following \cite{DG2,G1}, we define the \emph{pseudo-identity} functor 
\begin{equation}   \label{e:ps-id}
(\psId)_{\cY,!}:\Dmod (\cY)_c\to\Dmod (\cY)
\end{equation}
to be the functor  \eqref{e:functor-cF} corresponding to $\cF=\Delta_!(k_{\cY})$. We will also consider the pairing \eqref{e:pairing-cF} corresponding to $\cF=\Delta_!(k_{\cY})$.

The functor \eqref{e:ps-id} has the following important property, which can be checked
straightforwardly: for any open $U\overset{j}\mono\cY$ 
one has 
\begin{equation} \label{e:property of paid}
(\psId)_{\cY,!}\circ j_*=j_!\circ (\psId)_{U,!}\, .
\end{equation}

\begin{rem}
Suppose that $\cY$ is truncatable (this notion was defined in \remref{r: truncatability}). Then $\Dmod (\cY)_c$ is a full subcategory of $\Dmod (\cY)^\vee$, see formula~\eqref{e:Dc}. In fact, the functor \eqref{e:functor-cF} uniquely extends to a continuous functor
\begin{equation}   \label{e:extended by continuity}
\Dmod (\cY)^\vee\to \Dmod (\cY),
\end{equation}
and the pairing \eqref{e:pairing-cF} to a continuous pairing 
\[\Dmod (\cY)^\vee\times\Dmod (\cY)^\vee\to\Vect;
\]
see \cite[Subsect.~4.4.8]{DG2} or  \cite[Subsect.~3.1.1]{G1}.
\end{rem}

\begin{rem}  \label{r:psid to id}
Suppose that $\cY$ is smooth of pure dimension $d$ and that the morphism $\Delta:\cY\to\cY\times\cY$ is separated.\footnote{Sometimes (e.g., in \cite{LM}) separateness of $\Delta$ is required in the definition of algebraic stack. Anyway, for most stacks that appear in practice  
the morphism $\Delta$ is affine (and therefore separated).} Then one has a canonical morphism
\[
\Delta_!(k_\cY)\to\Delta_*(k_\cY)=\Delta_*(\omega_\cY)[-2d],
\]
which induces a canonical morphism
\begin{equation}   
(\psId)_{\cY,!}(M)\to M[-2d], \quad M\in\Dmod(\cY )_c\, .
\end{equation}
\end{rem}

\subsubsection{D-module analogs of  $\B^K\,$ and $L^K$}   \label{sss:analogs of BK&LK}
On $\cA_c$ we have the bilinear form $\B$; its restriction to $\cA_c^K$ will be denoted by $\B^K$. In Subsect.~\ref{ss:operator A} we defined
the operator $L:\cA_c\to\cA\,$; it induces an operator $L^K:\cA_c^K\to\cA^K$. Recall that
\[
\B (f_1,f_2)=\B_{naive}(f_1,Lf_2)
\]
for any $f_1,f_2\in\cA_c\,$; in particular, this is true for $f_1\, ,f_2\in\cA_c^K$.

According to Theorem~\ref{t:our main}, the `matrix' of the bilinear form $\B^K$ is the function $b$ defined in Subsect.~\ref{sss:b}.
According to Example~\ref{ex:b}, the D-module $\Delta_!(k_{\Bun_G})$ is an analog of the function $q^{-d}\cdot b$, where
\[
d:=\dim\Bun_G=3g_X-3\,.
\]
So we consider the pairing \eqref{e:pairing-cF} corresponding to $\cY=\Bun_G$ and $\cF=\Delta_!(k_{\cY})$ to be the D-module analog of the bilinear form $q^{-d}\cdot\B^K$. Accordingly, we consider the functor 
$(\psId)_{\Bun_G,!}:\Dmod (\Bun_G)_c\to\Dmod (\Bun_G)$ defined in Subsect.~\ref{sss:ps-id} to be the D-module analog of the operator 
$q^{-d}\cdot L^K:\cA_c^K\to\cA^K$.

\subsection{Miraculous duality and a D-module analog of \corref{c:A}}  \label{ss:miraculous}
Formula \eqref{e:property of paid} implies that for any stack $\cY$, the functor $(\psId)_{\cY,!}$ maps $\Dmod (\cY)_c$ to the full subcategory 
$\Dmod(\cY)_{ps-c}\subset\Dmod (\cY)$, so one gets a functor
\begin{equation}   \label{e:2ps-id}
(\psId)_{\cY,!}: \Dmod (\cY)_c\to\Dmod (\cY)_{ps-c}\, .
\end{equation}

Now suppose that $\cY=\Bun_G\,$. Then the main result of \cite{G1} (namely, Theorem 0.1.6) says that the functor \eqref{e:extended by continuity} 
corresponding to $\cF=\Delta_!(k_{\cY})$ is an equivalence.\footnote{In addition to the functor \eqref{e:extended by continuity}, one also has the naive functor $\Dmod (\cY)^\vee\to \Dmod (\cY)$, which extends the natural embedding 
$\Dmod (\cY)_c\mono\Dmod (\cY)$. But this naive functor is \emph{not} an equivalence by \cite[Lemma 4.4.5]{DG2}.}
By \cite[Lemma 4.5.7]{DG2}, this implies that the functor \eqref{e:2ps-id} is an equivalence. This is an analog of the part of \corref{c:A} that says that the operator $L^K:\cA_c\to\cA$ induces an isomorphism
 $\cA_c^K\iso\cA_{ps-c}^K\,$.

 \corref{c:A} also explicitly describes the inverse isomorphism 
 $$\cA_{ps-c}^K\iso\cA_c^K.$$
 This suggests a conjectural description of the functor 
 inverse to \eqref{e:2ps-id}. The conjecture is formulated in Appendix~\ref{s:conjectural}.

 \subsection{The functor $\psId_{\cY,!}$ for $\cY=\Bun_T\,$} \label{ss:A.9}
  The material of this subsection will be used in Subsect.~\ref{sss:Dennis-iso}.

\subsubsection{D-module setting}   \label{sss:Dsetting}
The morphism $\Delta:\Bun_T\to\Bun_T\times\Bun_T$ factors as 
$$\Bun_T\overset{\pi}\longrightarrow\cZ\overset{i}\mono\Bun_T\times\Bun_T\, ,$$ 
where $i$ is a closed embedding and $\pi :\Bun_T\to\cZ$ is a $\bG_m$-torsor. So $\Delta_! (k_{\Bun_T})=\Delta_* (k_{\Bun_T})[-1]$. The stack $\Bun_T$ is smooth and has pure dimension $g_X-1$, where $g_X$ is the genus of $X$. So $k_{\Bun_T}=\omega_{\Bun_T}[2-2g_X]$. Thus 
$$\Delta_! (k_{\Bun_T})=\Delta_* (\omega_{\Bun_T})[1-2g_X].$$ Therefore the functor $\psId_{\Bun_T,!}:\Dmod (\Bun_T )_c\to\Dmod (\Bun_T )_{ps-c}=\Dmod (\Bun_T )_c$  equals $\Id [1-2g_X]$.
 
 \subsubsection{$l$-adic setting}    \label{sss:lsetting}
 In this setting the formulas are similar to those from Subsect.~\ref{sss:Dsetting}, but now we have to take the Tate twists in account:
 \[
 \Delta_! ((\ql )_{\Bun_T})= \Delta_*((\ql )_{\Bun_T})[-1]= \Delta_*(\dK_{\Bun_T})[1-2g_X](1-g_X),
 \]
\begin{equation}   \label{e:l-psid}
\psId_{\Bun_T,!}=\Id [1-2g_X](1-g_X).
\end{equation}

 \subsubsection{Analog at the level of functions}   \label{sss:analog of psid-BunT}
 We consider the vector space $\C_c^K$ to be an analog of $\Dmod (\Bun_T )_c\,$. We consider the operator 
 \begin{equation}    \label{e:minus}
 -q^{1-g_X}\cdot\Id\in\End (\C_c^K) 
 \end{equation}
 to be an analog of the functor $\psId_{\Bun_T,!}:\Dmod (\Bun_T )_c\to\Dmod (\Bun_T )_c\,$. This is justified by formula~\eqref{e:l-psid}; in particular, the minus sign in \eqref{e:minus} is due to the fact that the number $1-2g_X$ from  \eqref{e:l-psid} is odd. In Subsect.~\ref{sss:Dennis-iso} we will see that this minus sign is closely related to the minus sign in \propref{p:properties of A}(ii).

\subsection{Eisenstein functors }  \label{ss:Eisfunctors}
The operators $$\Eis :\C_+\to\cA, \quad \Eis' :\C_-\to\cA$$ induce operators $\Eis^K :\C_+^K\to\cA^K$ and 
$(\Eis')^K :\C_-^K\to\cA^K$. In Subsections~\ref{sss:Eis*}-\ref{sss:Dennis-notation} we will discuss the functor $\Eis_*\,$, which is a D-module analog of the operator $\Eis^K$. In Subsections~\ref{sss:Eis!}-\ref{sss:Dennis-iso}  we will discuss the functor 
$\Eis_!\,$, whose analog at the level of functions  is closely related to $(\Eis')^K$, see 
formula~\eqref{e:analog of Eis!}. In Subsection~\ref{sss:other Eis-functors} we briefly discuss  the compactified Eisenstein functor $\Eis_{!*}$ and the enhanced Eisenstein functor $\Eis^{\enh}\,$.

Both  $\Eis_*\,$ and $\Eis_!\,$ are defined using the diagram of stacks
\[ 
\xy
(-15,0)*+{\Bun_G}="X";
(15,0)*+{\Bun_T}="Y";
(0,15)*+{\Bun_B}="Z";
{\ar@{->}_{\sfp} "Z";"X"};
{\ar@{->}^{\sfq} "Z";"Y"};
\endxy
\]
which was already used in Subsection~\ref{ss:CTAps-c}. We will need the following remarks.

\begin{rem} \label{r:representability}
The morphism $\sfp:\Bun_B\to\Bun_G$ is representable, i.e., its fibers are algebraic spaces (in fact, schemes).
\end{rem}

\begin{rem}   \label{r:qcompactness}
The morphism $\sfp:\Bun_B\to\Bun_G$ is not quasi-compact. But the restriction of $\sfp$ to the substack $\sfq^{-1}(\Bun_T^{\ge a} )$ is quasi-compact for any $a\in\bZ$. Here 
$$\Bun_T^{\ge a}\subset\Bun_T=\Bun_{\bG_m}$$
is the stack of $\bG_m$-bundles of degree $\ge a$.
\end{rem}

\subsubsection{The functor $\Eis_*\,$ as a D-module analog of the operator $\Eis^K$}   \label{sss:Eis*}
Let $\Dmod(\Bun_T)_+$ denote the full subcategory formed by those $M\in\Dmod(\Bun_T)$ whose support is contained in $\Bun_T^{\ge a}$ for some $a\in\bZ$. Define a functor 
\begin{equation}   \label{e:Eis*}
\Eis_*:\Dmod (\Bun_T)_+\to\Dmod (\Bun_G)
\end{equation}
by $\Eis_*:=\sfp_*\circ\sfq^!$. Since we consider $\Dmod (\Bun_T)_+$ rather than $\Dmod (\Bun_T)\,$, Remarks~\ref{r:representability}-\ref{r:qcompactness} ensure that taking $\sfp_*$ does not lead to pathologies.
It is easy to check that
\begin{equation}   \label{e:c-inclusion}
\Eis_*(\Dmod (\Bun_T)_c)\subset\Dmod (\Bun_G)_c\,.
\end{equation}

The functor $\Eis_*$ is a D-module analog of the operator $\Eis^K :\C_+^K\to\cA^K$, and formula~\eqref{e:c-inclusion} is similar to the inclusion
$\Eis^K (\C_c^K)\subset\cA_c^K$. This is clear from \remref{r:the_2_diagrams}. 

The reader may prefer to skip the next subsection and go directly to Subsect.~\ref{sss:Eis!}.

\subsubsection{Relation to the notation of \cite{G1}}   \label{sss:Dennis-notation} 

The functor \eqref{e:Eis*} is the restriction of the functor 
\begin{equation}   \label{e:2Eis*}
\Eis_*:\Dmod (\Bun_T)\to\Dmod (\Bun_G)
\end{equation}
defined in  \cite[Subsect.~1.1.9]{G1}.

On the other hand, the DG category $\Dmod (\Bun_G)_c$ is a full subcategory of the Lurie dual $\Dmod (\Bun_G)^\vee$, see \remref{r:skip} and especially formula~\eqref{e:Dc}.  The DG category $\Dmod (\Bun_G)^\vee$ has a realization introduced in \cite[Subsect.~4.3.3]{DG2}  (or  \cite[Subsect.~1.2.2]{G1} ) and denoted there by $\Dmod (\Bun_G)_{\on{co}}\,$; for us, $\Dmod (\Bun_G)_{\on{co}}$ is a synonym of $\Dmod (\Bun_G)^\vee$. The functor 
$\Eis_*:\Dmod (\Bun_T)_c\to\Dmod (\Bun_G)_c$ is the restriction of the functor 
\begin{equation}   \label{e:3Eis*}
\Dmod (\Bun_T)\to\Dmod (\Bun_G)^\vee=\Dmod (\Bun_G)_{\on{co}}
\end{equation}
introduced in  \cite{G1} and denoted there by $(\CT_* )^\vee$ or $\Eis_{\on{co},*}$ (the notation $(\CT_* )^\vee$ is  introduced in  \cite[Subsect.~0.1.7]{G1} and the synonym $\Eis_{\on{co},*}$ in \cite[Subsect.~1.4.1]{G1}).

Let us note that the relation between the functors \eqref{e:2Eis*} and  \eqref{e:3Eis*} is described in \cite[Prop.~2.1.7]{G1}.

\subsubsection{The functor $\Eis_!\,$}   \label{sss:Eis!}
Define a functor
$$\Eis_!:\Dmod (\Bun_T)\to\Dmod (\Bun_G)$$ 
by $\Eis_!:=\sfp_!\circ\sfq^*$.  According to \cite[Cor.~2.3]{DG2}, the functor $\Eis_!$ is defined everywhere.\footnote{This is not obvious because the functor $\sfp_!$ is only \emph{partially} defined. However, it is proved in  \cite{DG2} that $\sfp_!$ is defined on the essential image of $\sfq^*$. (The functor $\sfq^*$ is defined everywhere because $\sfq$ is smooth.)} 

Note that by \remref{r:qcompactness}, the restriction of $\Eis_!$ to $\Dmod (\Bun_T)_+$ preserves holonomicity. 
It is easy to check that
\begin{equation}   \label{e:2c-inclusion}
\Eis_!(\Dmod (\Bun_T)_c)\subset\Dmod (\Bun_G)_{ps-c}\,.
\end{equation}

\subsubsection{The analog of $\Eis_!$ at the level of functions}   \label{sss:analog of Eis!}
First, let us define a certain automorphism of the space $\C^K$. As explained in Example~\ref{ex:functional},  $\C^K$ identifies with the space of functions on $\Pic X$, where $X$ is the smooth projective curve over $\bF_q$ corresponding to the global field $F$.  So the inversion map $\iota :\Pic X\to\Pic X$ induces an operator $\iota^*:\C^K\to\C^K$, which interchanges the subspaces $\C_+^K$ and $\C_-^K\,$.

Now we claim that \emph{the functor $\Eis_!:\Dmod (\Bun_T)_+\to\Dmod (\Bun_G)$ is a D-module analog of the operator}
\begin{equation}  \label{e:analog of Eis!}
q^{2-2g_X}\cdot (\Eis')^K\circ \iota^* :\C_+^K\to\cA^K\,,
\end{equation}
where $\Eis':\Cm\to\cA$ is the operator defined in Subsection~\ref{ss:second Eis}.
This claim is justified by Theorem~\ref{t: Eis_! and Eis'} of Appendix~\ref{s: Eis_! and Eis'}. In Subsect.~\ref{sss:Dennis-iso} below we show that this claim agrees with Theorem 4.1.2 of \cite{G1}; this gives another justification.

Note that by \propref{p:Eis'(C_c)}, the operator 	\eqref{e:analog of Eis!} maps $\C_c^K$ to $\cA_{ps-c}^K\,$. This is similar to the inclusion \eqref{e:2c-inclusion}.

\subsubsection{Comparison with  \cite[Theorem 4.1.2]{G1}}    \label{sss:Dennis-iso}
Theorem 4.1.2 of \cite{G1} tells us\footnote{To see this, use Subsection~\ref{sss:Dennis-notation} and the fact that the composition 
$\Eis_*\circ\, \iota^*$ (which appears in formula \eqref{e:Dennis-rhs}) equals the functor $ \Eis^-_{\on{co},*}$ (which appears in \cite[Theorem 4.1.2]{G1}). Note that if $G$ is an \emph{arbitrary} reductive group rather than $SL(2)$ then $\iota^*$ has to be replaced here by $w_0^*$, where $w_0$ is the longest element of the Weyl group.} that the functor
\[
\Eis_!\circ \psId_{\Bun_T,!}:\Dmod (\Bun_T)_c\to\Dmod (\Bun_G)_{ps-c}
\]
is isomorphic to the functor
\begin{equation}   \label{e:Dennis-rhs}
\psId_{\Bun_G,!}\circ \Eis_*\circ\, \iota^* :\Dmod (\Bun_T)_c\to\Dmod (\Bun_G)_{ps-c}\,,
\end{equation}
where $\iota:\Bun_T\iso\Bun_T$ is the inversion map. By Subsections~\ref{sss:analogs of BK&LK} and \ref{sss:Eis*}, the analog of \eqref{e:Dennis-rhs} at the level of functions is the operator 
$$q^{3-3g_X}\cdot L^K\circ\Eis^K\circ\, \iota^*:\C_c^K\to\cA_{ps-c}^K\,.$$
This operator equals $-q^{3-3g_X}\cdot (\Eis')^K\circ \iota^*$ by \propref{p:properties of A}(ii). By Subsect.~\ref{sss:analog of psid-BunT}, the analog of the functor 
$\psId_{\Bun_T,!}$ at the level of functions is the operator of multiplication by $-q^{1-g_X}$. So we see that 
the claim made in Subsect.~\ref{sss:analog of Eis!} agrees with  Theorem 4.1.2 of \cite{G1}.

\subsubsection{Other Eisenstein functors}   \label{sss:other Eis-functors}
The functor $\Eis_*$ has an `enhanced' version $\Eis^{\enh}=\Eis_B^{\enh}$, see Subsection~\ref{ss:enh}. Both $\Eis_*$ and $\Eis^{\enh}$ are D-module analogs of the operator~$\Eis^K$.

One also has the compactified Eisenstein functor $\Eis_{!*}\,$, see Subsect.~\ref{ss:why barbaric}(i). In Appendix~\ref{s: Eis_! and Eis'} we work with slightly different functors $\overline{\Eis}$ and $\underline{\Eis}$ (see Subsect.~\ref{ss:def of compactified}), which are good enough for $G=SL(2)$. Formulas~\eqref{e:B4}-\eqref{e:B5} from \corref{c:relation between Eis's} describe the analogs of $\overline{\Eis}$ and $\underline{\Eis}$ at the level of functions. In terms of Eisenstein series (rather than Eisenstein operators) the functor $\Eis_{!*}$ corresponds to the product of the Eisenstein series by a  normalizing factor, which is essentially an $L$-function in the case $G=SL(2)$ and a product of $L$-functions in general (see \cite[Theorem 3.3.2]{Lau} and \cite[Subsect.~2.2]{BG} for more details).

 \section{Relation between the functor $\Eis_!$ and the operator $(\Eis')^K$}  \label{s: Eis_! and Eis'}
 In this section we work over $\bF_q\,$.  Our main goal is to prove Theorem~\ref{t: Eis_! and Eis'}, which justifies the
 claim made in Subsect.~\ref{sss:analog of Eis!}. 
   
 \subsection{Notation and conventions}
\subsubsection{}
 We will say `stack' instead of `algebraic stack locally of finite type over $\bF_q\,$'.

 \subsubsection{}
 Let $X$ be a smooth complete geometrically connected curve over $\bF_q\,$. 
Just as in the rest of the article, $G:=SL(2)$, $T\subset G$ is the group of diagonal matrices, and  $B\subset G$ is the subgroup of upper-triangular matrices.  Let $\Bun_G$ (resp.~ $\Bun_T$) denote the moduli stack of principal $G$-bundles (resp.~$T$-bundles) on $X$.

 \subsubsection{}   \label{sss:2Functions-sheaves}
 We fix a prime $l$ not dividing $q$ and an algebraic closure $\ql$ of $\bQ_l\,$. For any stack $\cY$ one has the  bounded constructible derived category of $\ql$-sheaves, denoted by $\sD (\cY )$. 
 
 To any $\cF\in\sD (\cY )$ we associate a function $f_\cF:\cY (\bF_q)\to\ql$ using the (nonstandard) convention of Subsection~\ref{sss:Functions-sheaves}: namely, the value of $f_\cF$ at $y\in\cY (\bF_q)$ is the trace of the \emph{arithmetic} Frobenius acting on the \emph{!-stalk} of $\cF$ at $y$. So the standard operators between spaces of functions correspond to the \emph{`right'} functors in the sense of Subsect.~\ref{sss:left&right}. 
 
 \subsubsection{}   \label{sss:K-theory}
 If a stack $\cY$ is quasi-compact, let $\K (\cY )$ denote the Grothendieck group of $\sD (\cY )$. In general, let $\K (\cY )$ denote the projective limit of the groups $\K(U)$ corresponding to quasi-compact open substacks $U\subset\cY$. We equip $\K(\cY )$ with the projective limit topology. 
 
 The assignment $\cF\mapsto f_\cF$ from Subsect.~\ref{sss:2Functions-sheaves} clearly yields a group homomorphism from $\K (\cY )$ to the space of functions $\cY (\bF_q)\to\ql$. This homomorphism is still denoted by $\cF\mapsto f_\cF\,$.

  \subsubsection{}
 Just as in Subsect.~\ref{ss:Eisfunctors}, we consider the diagram of stacks
\begin{equation} \label{e2basic diagram}
\xy
(-15,0)*+{\Bun_G}="X";
(15,0)*+{\Bun_T.}="Y";
(0,15)*+{\Bun_B}="Z";
{\ar@{->}_{\sfp} "Z";"X"};
{\ar@{->}^{\sfq} "Z";"Y"};
\endxy
\end{equation}
that comes from the diagram of groups $G\hookleftarrow B\twoheadrightarrow T$. The morphism $\sfp:\Bun_B\to\Bun_G$ is representable. It is not quasi-compact, but the restriction of $\sfp$ to the substack $\sfq^{-1}(\Bun_T^{\ge a} )$ is quasi-compact for any $a\in\bZ$. So we have functors 
\begin{equation}   \label{2e:Eis*}
\Eis_*:\sD (\Bun_T)_+\to\sD (\Bun_G), \quad \Eis_*:=\sfp_*\circ\sfq^!,
\end{equation}
\begin{equation}   \label{2e:Eis!}
\Eis_!:\sD (\Bun_T)_+\to\sD (\Bun_G), \quad \Eis_!:=\sfp_!\circ\sfq^*,
\end{equation}
where $\sD (\Bun_T)_+$ denotes the full subcategory formed by those $\cF\in\sD (\Bun_T)$ whose support is contained in $\Bun_T^{\ge a}$ for some $a\in\bZ$.

\subsubsection{}
Let $\K(\Bun_T)_+$ denote the direct limit of $\K (\Bun_T^{\ge a})$, $a\in\bZ$. The functors \eqref{2e:Eis*}-\eqref{2e:Eis!} induce group homomorphisms
\[
\Eis_*:\K (\Bun_T)_+\to\K (\Bun_G), \quad \Eis_!:\K (\Bun_T)_+\to\K (\Bun_G).
\]

 \subsubsection{}
 Let $F$ denote the field of rational functions on $X$ and $\bA$ its adele ring. Recall that $K\subset G(\bA)$ denotes the standard maximal compact subgroup.
 
 Let $\cA$ and $\C$, $\Cp$, $\Cm$ be the functional spaces defined in Subsections~\ref{sss:automorphic} and \ref{ss:C}; we take $\ql$ as the field in which our functions take values.
 
 As explained in Example~\ref{ex:functional}, we identify $\cC^K$ (i.e., the subspace of $K$-invariants in $\cC$) with the space of all $\ql$-valued functions on $\Bun_T(\bF_q)=\Bun_{\bG_m}(\bF_q)$. We identify $\cA^K$ with the space of all functions $\Bun_G(\bF_q)\to\ql\,$.

 \subsubsection{}   \label{sss:operators}
 The inversion map $\iota :\Bun_T\to\Bun_T$ induces an operator $\iota^*:\C^K\to\C^K$, which interchanges the subspaces 
 $\C_+^K$ and $\C_-^K\,$.
 
 The operators $\Eis:\Cp\to\cA$ and $\Eis':\Cm\to\cA$ defined in Subsections~\ref{ss:Eis} and \ref{ss:second Eis} induce operators $\Eis^K:\Cp^K\to\cA^K$ and $(\Eis')^K:\Cm^K\to\cA^K$.
 Since $\iota^* (\C^K_+)=\C^K_-$ we get an operator
\begin{equation} \label{e:2analog of Eis!}
(\Eis')^K\circ \iota^* :\C_+^K\to\cA^K\,.
\end{equation}

 \subsection{Formulation of the theorem}
 For any $\cF\in\K (\Bun_T)_+$ one has the functions 
 $$f_\cF\in\Cp^K, \quad f_{\Eis_*\cF}\in\cA^K, \quad f_{\Eis_!\cF}\in\cA^K$$ 
 defined as explained in Subsections~\ref{sss:2Functions-sheaves}-\ref{sss:K-theory}. Since formula~\eqref{2e:Eis*} involves only `right' functors, it is clear that
\[
f_{\Eis_*\cF}=\Eis^K(f_\cF), \quad \cF\in\K (\Bun_T)_+ \; .
\]
The next theorem expresses $f_{\Eis_!\cF}$ in terms of $f_\cF$ and the operator \eqref{e:2analog of Eis!}.

 \begin{thm}       \label{t: Eis_! and Eis'}
 For any $\cF\in\K (\Bun_T)_+$ one has
\begin{equation}   \label{e:to prove}
f_{\Eis_!\cF}= q^{2-2g_X}\cdot ((\Eis')^K\circ \iota^*)(f_\cF)\, ,
\end{equation}
where $g_X$ is the genus of $X$.
\end{thm}

A proof is given in Subsections~\ref{ss:compactified Eisenstein}-\ref{ss:at last the proof} below. To make it self-contained, we used an approach which is somewhat barbaric (as explained in Subsect.~\ref{ss:why barbaric}).

\begin{rem}
Using \cite[Cor.~4.5]{BG2}, one can express $f_{\Eis_!\cF}$ in terms of $f_\cF$ for \emph{any} reductive group $G$ (at least, in the case of principal Eisenstein series).
\end{rem}

\subsection{The compactified Eisenstein functors}   \label{ss:compactified Eisenstein}
We will need the `compactified Eisenstein' functors
$\overline{\Eis}, \underline{\Eis}:\sD (\Bun_T)_+\to\sD (\Bun_G)$,
which go back to \cite{Lau}.

\subsubsection{Definition of $\overline{\Eis}$ and $\underline{\Eis}$} \label{ss:def of compactified}
Consider the diagram
\[
\xy
(-20,0)*+{\Bun_G}="X";
(20,0)*+{\Bun_T}="Y";
(0,20)*+{\overline\Bun_B}="Z";
(-20,20)*+{\Bun_B}="W";
{\ar@{->}_{\overline\sfp} "Z";"X"};
{\ar@{->}^{\overline\sfq} "Z";"Y"};
{\ar@{^{(}->}^j "W";"Z"};  
\endxy
\]
in which $\overline\Bun_B$ denotes\footnote{The definition of $\overline\Bun_B$ is so simple because we assume that $G=SL(2)$. In the case of an arbitrary reductive group see \cite[Subsect.~1.2]{BG} } the stack of rank 2 vector bundles $\cL$ on $X$ with trivialized determinant equipped with an invertible subsheaf $\cM\subset\cL\,$ (the open substack $\Bun_B\subset\overline\Bun_B$ is defined by the condition that $\cM$ is a subbundle). Note that  the morphism $\overline\sfp:\overline\Bun_B\to\Bun_G$ is representable and its restriction  to the substack $\overline\sfq^{-1}(\Bun_T^{\ge a} )\subset\overline\Bun_B$ is proper for any $a\in\bZ$. 

Now define the functor $\overline{\Eis}:\sD (\Bun_T)_+\to\sD (\Bun_G)$ and the group homomorphism 
$\overline{\Eis}:\K (\Bun_T)_+\to\K (\Bun_G)$ by 
$$\overline{\Eis}:=\overline\sfp_*\circ\overline\sfq^!.$$

Similarly, define the functor $\underline{\Eis}:\sD (\Bun_T)_+\to\sD (\Bun_G)$ and the group homomorphism 
$\underline{\Eis}:\K (\Bun_T)_+\to\K (\Bun_G)$ by $$\underline{\Eis}:=\overline\sfp_!\circ\overline\sfq^*.$$

\subsubsection{Relation between $\overline{\Eis}$ and $\underline{\Eis}$}
Recall that the restriction  of $\overline\sfp:\overline\Bun_B\to\Bun_G$ to the substack $\overline\sfq^{-1}(\Bun_T^{\ge a} )\subset\overline\Bun_B$ is proper, so $\overline\sfp_!=\overline\sfp_*\,$.  On the other hand, the following (well known) fact implies that $\overline\sfq^*$ differs from 
$\overline\sfq^!$ only by a cohomological shift and a Tate twist.

\begin{prop}  \label{p:smoothness}
As before, assume that $G=SL(2)$. Then

(i) the morphism $\overline\sfq :\overline\Bun_B\to\Bun_T$ is smooth.

(ii) the fiber of $\overline\sfq$ over $\cM\in\Bun_T$ has pure dimension $-\chi (\cM^{\otimes 2})=g_X-1-2 \deg\cM\,$.
\end{prop}

We skip the proof because it is quite similar to that of \cite[Cor. 2.10]{Lau1}.

\begin{cor}    \label{c:Tate-twist}
One has
$$\underline{\Eis}(\cF)=\overline{\Eis} (\cF [2m](m)),\quad \cF\in\sD (\Bun_T)_+\, ,$$ 
where $m:\Bun_T\to\bZ$ is the locally constant function whose value at $\cM\in\Bun_T$ equals
$2 \deg\cM+1-g_X\,$. 
\end{cor}

\subsubsection{Expressing $\overline{\Eis}$ and $\underline{\Eis}$ in terms of  $\Eis_*\,$and $\Eis_!\,$}
The next proposition describes the relation between $\overline{\Eis}$ and $\Eis_*$ and a similar relation between $\underline{\Eis}$ and $\Eis_!$ at the level of Grothendieck groups. To formulate it, we need some notation.

Let $\Sym X$ denote the scheme parametrizing all effective divisors on $X$; in other words, $\Sym X$ is the disjoint union of $\Sym^nX$ for all $n\ge 0$.  Note that $\Sym X$ is a monoid with respect to addition. The morphism
\begin{equation}   \label{e:action map}
\on{act}:\Sym X\times\Bun_T\to\Bun_T\, , \quad (D,\cM)\mapsto\cM (-D)
\end{equation}
defines an action of the monoid $\Sym X$ on $\Bun_T\,$. Let
\[
\on{pr}:\Sym X\times\Bun_T\to\Bun_T
\]
denote the projection.

\begin{prop}  \label{p:K-theoretic Eisensteins}
(i)  The map $\overline{\Eis}:\K (\Bun_T )_+\to \K (\Bun_B )$ equals $\Eis_*\circ\on{pr}_*\circ\on{act}^!$.

(ii) The map $\underline{\Eis}:\K (\Bun_T )_+\to \K (\Bun_B )$ equals $\Eis_!\circ\on{pr}_!\circ\on{act}^*$.
\end{prop}

\begin{rem} \label{r:slight}
$\Sym^n X$ is proper for each $n$, so $\on{pr}_!=\on{pr}_*\,$. On the other hand, the morphism \eqref{e:action map} is smooth, so 
$\on{act}^*$ only slightly differs from $\on{act}^!\,$; more precisely, for any $\cF\in\sD (\Bun_T)_+$ the restrictions of $\on{act}^*(\cF )$ and 
$\on{act}^!(\cF )[-2n](-n)$ to $\Sym^nX\times\Bun_T$ are canonically isomorphic.
\end{rem}

\begin{proof}[Proof of \propref{p:K-theoretic Eisensteins}]
The proof given below is straightforward because statement~(i) involves only `right' functors and statement (ii) only `left' ones.

First, let us recall the standard stratification of $\overline{\Bun}_B\,$. If $\cL$ is a rank 2 vector bundle on $X$ with trivialized determinant, $\cM\subset\cL$ is a line sub-bundle, and $D\subset X$ is an effective divisor of degree $n$ then the pair $(\cL,\cM (-D))$ defines an $\bF_q$-point of $\overline{\Bun}_B\,$. This construction works for $S$-points instead of $\bF_q$-points. It defines a locally closed immersion 
\[
i_n:\Sym^nX\times\Bun_B\mono\overline{\Bun}_B\, .
\]
The substacks $i_n(\Sym^nX\times\Bun_B)$ form a stratification of $\overline{\Bun}_B\,$.

Now let us prove (i). We have to check the equality 
\begin{equation}  \label{e:prove this}
\overline\sfp_*\circ\overline\sfq^!=\sfp_*\circ\sfq^!\circ\on{pr}_*\circ\on{act}^!\, ,
\end{equation}
in which both sides are maps $\K (\Bun_T)_+\to\K (\Bun_G)$. For any $\cF\in\K (\overline{\Bun}_B )$ one has
\[
\cF=\sum_{n=0}^{\infty}(i_n)_*\circ i_n^! (\cF)
\]
(the sum converges in the topology of $\K (\overline{\Bun}_B )$ defined in Subsect.~\ref{sss:K-theory}). So
\begin{equation}  \label{e:by stratification}
\overline\sfp_*\circ\overline\sfq^!=\sum_{n=0}^{\infty}(\overline\sfp\circ i_n)_*\circ (\overline\sfq\circ i_n)^!\, , 
\end{equation}
To see that the right hand sides of \eqref{e:prove this} and \eqref{e:by stratification} are equal, it suffices to apply base change to the expression $\sfq^!\circ\on{pr}_*$ from the r.h.s. of \eqref{e:prove this}.

We have proved (i). Statement (ii) can be either proved similarly or deduced from (i) by Verdier duality.
\end{proof}

\subsection{Passing from sheaves to functions}
 Recall that we think of $\cC^K$ as the space of $\ql$-valued functions on $\Bun_T(\bF_q)=\Bun_{\bG_m}(\bF_q)$ (see Example~\ref{ex:functional}). 
 
\begin{lem}    \label{l:pract}
As before, let $\on{pr}:\Sym X\times\Bun_T\to\Bun_T$ denote the projection and $\on{act}:\Sym X\times\Bun_T\to\Bun_T$ the morphism $(D,\cM)\mapsto\cM (-D)$.

(i) One has commutative diagrams

\begin{equation}   \label{e:diagram}
\CD
\K (\Bun_T)_+  @>{\on{pr}_*\circ\on{act}^!}>>\K (\Bun_T)_+ @.{\phantom{aaa}}
\K (\Bun_T)_+  @>{\on{pr}_!\circ\on{act}^*}>>\K (\Bun_T)_+ \\
@VVV    @VVV  @VVV    @VVV \\
\Cp^K  @>{\U_0}>>\Cp^K @. \Cp^K  @>{\U_{-1}}>>\Cp^K
\endCD
\end{equation}
in which each vertical arrow is the map $\cF\mapsto f_{\cF}$ and the operator $\U_n: \Cp^K\to \Cp^K $ is defined by 
\begin{equation}  \label{e:B2}
(\U_n \varphi )(\cM)=\sum_{D\ge 0} q^{n\cdot\deg D}\varphi (\cM (-D)), \quad \varphi\in\Cp^K
\end{equation}
(summation over all effective divisors on $X$).

(ii) All horizontal arrows in diagrams \eqref{e:diagram} are invertible.
\end{lem}

\begin{proof}
Statement (i) is clear (in the case of $\on{pr}_!\circ\on{act}^*$ use \remref{r:slight}). 

Let us prove that the operator $\U_n:\Cp^K\to\Cp^K$ is invertible (invertibility of the upper horizontal arrows is proved similarly). The space  $\Cp^K$ is complete with respect to the filtration formed by the subspaces $\C^K_{\ge\cN}\subset\Cp^K$, $\cN\in\bZ$. The operator $\U_n$ is compatible with the filtration and acts as identity on the successive quotients. So $\U_n$ is invertible.
\end{proof}

Recall that one has the operator $\Eis^K:\Cp^K\to\cA^K$.

\begin{cor}    \label{c:relation between Eis's}
For any $\cF\in\K (\Bun_T)_+$ one has 
\begin{equation}  \label{e:B3}
f_{\Eis_*\cF}=\Eis^K(f_\cF), \quad \cF\in\K (\Bun_T)_+ \; ,
\end{equation}
\begin{equation}  \label{e:B4}
f_{\overline\Eis\,\cF}=(\Eis^K\circ\U_0) (f_\cF), \quad \cF\in\K (\Bun_T)_+ \; ,
\end{equation}
\begin{equation}  \label{e:B5}
f_{\underline\Eis\, \cF}=(\Eis^K\circ\U_0\circ \W )(f_\cF), \quad \cF\in\K (\Bun_T)_+ \; ,
\end{equation}
\begin{equation}  \label{e:B6}
f_{\Eis_!\cF}=(\Eis^K\circ \U_0\circ\U_1^{-1}\circ \W)(f_\cF), \quad \cF\in\K (\Bun_T)_+ \; ,
\end{equation}
where $\U_0\,  ,\U_1:\Cp ^K\to\Cp ^K$ are defined by formula\, \eqref{e:B2} and  
$\W:\Cp ^K\to\Cp ^K$ is the operator of multiplication by the function 
\begin{equation}   \label{e:Euler}
\cM\mapsto q^{2\deg\cM+1-g_X}, \quad \cM\in\Bun_T (\bF_q).
\end{equation}
\end{cor}

\begin{proof}
Formula~\eqref{e:B3} is clear because the definition of $\Eis_*$ involves only `right' functors. Let us prove \eqref{e:B6} (the proof of \eqref{e:B4}-\eqref{e:B5} is similar but easier).

By \propref{p:K-theoretic Eisensteins} and  \corref{c:Tate-twist}, the map $\Eis_!:\K (\Bun_T)_+\to\K (\Bun_G)$ equals 
$\Eis_*\circ\widetilde \U_0\circ\widetilde\W\circ (\widetilde\U_{-1})^{-1}$, where $\widetilde\U_0:=\on{pr}_*\circ\on{act}^!$, $\widetilde \U_{-1}:=\on{pr}_!\circ\on{act}^*$, and 
$\widetilde\W (\cF):=\cF [2m](m)$ (here $m$ is as in \corref{c:Tate-twist}); note that $\widetilde\U_{-1}$ is invertible by \lemref{l:pract}(ii). 
 \lemref{l:pract}(i) and formula~\eqref{e:B3} imply that
 \[
 f_{\Eis_!\cF}=(\Eis^K\circ \U_0\circ \W\circ(\U_{-1})^{-1})(f_\cF).
 \]
 This is equivalent to \eqref{e:B6} because $\W^{-1}\circ\U_1\circ\W=\U_{-1}\,$.
\end{proof}

\subsection{Proof of Theorem~\ref{t: Eis_! and Eis'}}   \label{ss:at last the proof}
Theorem~\ref{t: Eis_! and Eis'} involves the operator 
\begin{equation}   \label{e:thoperator}
q^{2-2g_X}\cdot (\Eis')^K\circ \iota^*:\Cp^K\to\cA^K\, .
\end{equation} 
By definition,  $ (\Eis')^K=\Eis^K\circ (M^K)^{-1}$. Formulas \eqref{e:needed}-\eqref{e:2needed} tell us that 
$$(M^K)^{-1}\circ\iota^*=q^{2g_X-2}\cdot\U_0\circ\U_1^{-1}\circ \W,$$
where $\W:\Cp^K\to\Cp^K$ is the operator of multiplication by the function \eqref{e:Euler}.
So the operator \eqref{e:thoperator} is equal to the operator $\Eis^K\circ \U_0\circ\U_1^{-1}\circ \W$, which appears in formula~\eqref{e:B6}. \qed

\subsection{Concluding remarks}   \label{ss:why barbaric}
The above proof of Theorem~\ref{t: Eis_! and Eis'} is self-contained. On the other hand, it is barbaric for the following reasons.

\medskip

(i) We heavily used smoothness of the morphism $\overline\sfq :\overline\Bun_B\to\Bun_T\,$, which is a specific feature of the case $G=SL(2)$.
If $G$ is an arbitrary reductive group then instead of $\overline{\Eis}$ and $\underline{\Eis}$ one should work with the functor 
$\Eis_{!*}:\sD (\Bun_T)_+\to\sD (\Bun_G)$ introduced by A.~Braverman and D.~Gaitsgory \cite[Subsect.~2.1]{BG} (they denote it simply by $\Eis$; the notation $\Eis_{!*}$ is taken from \cite{G1}). In the case $G=SL(2)$ the functor $\Eis_{!*}$ is the `geometric mean' of our functors $\overline{\Eis}$ and $\underline{\Eis}$.

\smallskip

(ii) Our \propref{p:K-theoretic Eisensteins} is a statement at the level of $\K$-groups (and so is the more general Corollary~4.5 from \cite{BG2}). However, as explained to me by D.~Gaistgory, there is a way to relate the functors $\Eis_{!*}\,$, $\Eis_{!}\,$, and $\Eis_{*}$ themselves
(not merely the corresponding homomorphisms of $\K$-groups). His formulation of the relation involves the factorization algebras $\Upsilon$ and 
$\Omega$ introduced in \cite[Subsects.~3.1 and 3.5]{BG2}. One can think of these algebras as geometrizations of the operators $\U_0^{-1}$ and $\U_1^{-1}$, where $\U_n$ is defined by~\eqref{e:B2}. To make the analogy more precise, one should think of $\U_n$ not as an operator but as an element of the algebra $A$ from Subsect.~\ref{sss:algebra A}. The fact that $\Upsilon$ and $\Omega$ are \emph{factorization} algebras is related to the \emph{Euler product} expression for $\U_n$ in formula~\eqref{e:Upsilon}.

\section{A conjectural D-module analog of formula~\eqref{e:Ainverse}.} \label{s:conjectural}

Recall that according to \corref{c:A}, in the case of function fields the operator $L:\cA_c\to\cA_{ps-c}$ is invertible and its inverse is given by formula~\eqref{e:Ainverse}. In Subsections~\ref{ss:psid}-\ref{ss:miraculous} we defined a functor 
$$(\psId)_{\Bun_G,!}: \Dmod (\Bun_G)_c\to\Dmod (\Bun_G)_{ps-c}\, ;$$ 
as explained in Subsect.~\ref{sss:analogs of BK&LK}, this functor is a D-module analog of the operator 
\[
q^{-d}L^K:\cA_c^K\to\cA_{ps-c}^K\; , \quad d:=\dim\Bun_G\, .
\]
 As explained in  Subsection~\ref{ss:miraculous}, the main theorem of \cite{G1} implies that this functor is invertible. Conjecture~\ref{conj:conjectural} below gives a description of the inverse functor, which is inspired by formula~\eqref{e:Ainverse}. Before formulating the conjecture, we have to define
a certain endofunctor of $\Dmod (\Bun_G)$, which can be considered as a D-module analog of the operator $1-\Eis\circ\CT$ from the r.h.s of formula~\eqref{e:Ainverse}.

\subsection{An endofunctor of $\Dmod (\Bun_G)$}  \label{ss:enh}
Conjecture~\ref{conj:conjectural} involves the DG category $\I (G,B)$ defined in \cite[Sect.~6]{G2} and the adjoint pair of functors
$$\Eis_B^{\enh}:\I (G,B)\to\Dmod (\Bun_G), \quad \CT_B^{\enh}:\Dmod (\Bun_G)\to\I (G,B)$$ defined in \cite[Subsect.~6.3]{G2}; here `enh' stands for `enhanced'. The ideas behind these definitions are explained in  \cite[Subsect.~1.4]{G2}. More details regarding $\I (G,B)$, $\Eis_B^{\enh}$, and $\CT_B^{\enh}$ are contained in \cite[Subsects. 7.1, 7.3.5, 8.2.4]{Sn}.

 One can think of  $\I (G,B)$ as a `refined version' of the DG category $\Dmod (\Bun_T )$. More precisely, the DG category $\I (G,B)$ has a filtration indexed by integers whose associated graded equals $\Dmod (\Bun_T )$ (the grading on $\Dmod (\Bun_T )$ comes from the degree map 
 $\Bun_T\to\bZ$).\footnote{One can think of $\I (G,B)$ as the DG category of D-modules on a certian `stack' (in a generalized sense) equipped with a stratification whose strata are the stacks $\Bun_T^n\,$, $n\in\bZ$. This philosophy is explained in \cite{G2} (see \cite[Sect.~1.4]{G2}, which refers to \cite[Sect.~1.3.1]{G2}, which refers to \cite[Sect.~5]{G2}).} Both $\I (G,B)$ and $\Dmod (\Bun_T )$ are D-module analogs\footnote{This is not surprising in view of the previous footnote.} (in the sense of Subsect.~\ref{sss:Dmods&functions}) of the vector space $\C^K$. 
 
According to \cite[Subsect.~8.2.4]{Sn}, the functor  $\Eis_B^{\enh}:\I (G,B)\to\Dmod (\Bun_G)$ is left adjoint to $\CT_B^{\enh}:\Dmod (\Bun_G)\to\I (G,B)$. 
Let $\epsilon :\Eis_B^{\enh}\circ\CT_B^{\enh}\to \Id_{\Dmod (\Bun_G)}$ denote the counit of the adjunction. We need its cone, which is a functor
 \begin{equation}   \label{e:cone-eps}
 \Cone (\epsilon ) :\Dmod (\Bun_G )\to\Dmod (\Bun_G ).
 \end{equation}

\begin{rem}   \label{r:thinking}
We think of $\Eis_B^{\enh}$ and $\CT_B^{\enh}$ as D-module analogs of the operators $\Eis^K$ and $\CT^K$. We think of  $\Cone (\epsilon )$ as a D-module analog of the operator $1-\Eis^K\circ\CT^K$.
\end{rem}

 \subsection{The conjecture}
 Consider the composition
 \begin{equation}   \label{e:composed_funct}
\Dmod (\Bun_G)_{ps-c}\to\Dmod (\Bun_G)_c\mono\Dmod (\Bun_G)\, ,
 \end{equation}
 where the first functor is $((\psId)_{\Bun_G,!})^{-1}:\Dmod (\Bun_G)_{ps-c}\to\Dmod (\Bun_G)_c\,$. The following conjecture expresses this composition in terms of the functor \eqref{e:cone-eps}.
\begin{conj}   \label{conj:conjectural}
The composition \eqref{e:composed_funct} is  isomorphic to the restriction of the functor $\Cone (\epsilon )[2d]$ to $\Dmod (\Bun_G)_{ps-c}\,$, where
$d:=\dim\Bun_G\,$.
 \end{conj}
 
 \begin{rem}
 One can prove that the functor \eqref{e:cone-eps} indeed maps the subcategory $\Dmod (\Bun_G)_{ps-c}\subset\Dmod (\Bun_G)$ to $\Dmod (\Bun_G)_c$ (as would follow from the conjecture).
  \end{rem}

 \begin{rem}
 Let us compare the above conjecture with formula \eqref{e:Ainverse}. The functor $(\psId)_{\Bun_G,!}$ is a D-module analog of the operator 
 $q^{-d}L^K$. Formula \eqref{e:Ainverse} tells us that $(q^{-d}L^K)^{-1}=q^d (1-\Eis^K\circ\CT^K)$. This agrees with 
 Conjecture~\ref{conj:conjectural} by Remarks~\ref{r:thinking} and \ref{r:Tate twist}.
  \end{rem}
  
 \begin{rem} 
 Conjecture~\ref{conj:conjectural} implies that for any $N\in\Dmod(\Bun_G )_{ps-c}$ one has a canonical morphism 
 $N[2d] \to ((\psId)_{\Bun_G,!})^{-1}(N)$. This agrees with \remref{r:psid to id}, which says that for any $M\in\Dmod(\Bun_G )_c$ one has a canonical morphism $$(\psId)_{\Bun_G,!}(M)\to M[-2d].$$
\end{rem}

\bibliographystyle{alpha}

\end{document}